%% file: main.tex
\documentclass{amsart}
\usepackage{Preamble}
\title[Bounding The Orlov Spectrum]{Bounding The Orlov Spectrum \\For A Completion Of Discrete Cluster Categories}
\author{Dave Murphy}
\date{}

\begin{document}

\input{Sections/Abstract}

\maketitle

\input{Sections/Introduction}

\input{Sections/Prelims}

\input{Sections/ThickSubcats}

\input{Sections/Generators}

\input{Sections/GenTime}

\printbibliography

\end{document}

%% file: Sections/Abstract.tex
\begin{abstract}
We classify thick subcategories in a Paquette-Y\i ld\i r\i m completion $\overline{\cC}$ of a discrete cluster category of Dynkin type $A_{\infty}$.
To do this we introduce the notion of homologically connected objects, and the hc (=homologically connected) decomposition of an object into homologically connected objects in a $\mathrm{Hom}$-finite, Krull-Schmidt triangulated category.
We show that any object in a $\overline{\cC}$ has a hc decomposition, and that the hc decomposition determines the thick closure of an object.
Moreover, we use this result to classify the classical generators of $\overline{\cC}$ as homologically connected objects satisfying a maximality condition.

Every homologically connected object has an invariant, known as the homological length, and we show that in $\overline{\cC}$ this homological length is an upper bound for the generation time of a classical generator.
This allows us to provide an upper bound for the Orlov spectrum of $\overline{\cC}$, as well as giving the Rouquier dimension.
\end{abstract}

%% file: Sections/Introduction.tex
\section{Introduction}

Igusa and Todorov introduced discrete cluster categories of Dynkin type $A_{\infty}$, \cite{Igusa2013}, as a stable Frobenius category coming from a category defined via cyclic posets.
They are a nice class of 2-Calabi-Yau triangulated category, and generalise the cluster categories of Dynkin type $A_n$ introduced by Caldero, Chapoton and Schiffler \cite{CCS}, and more generally by Buan, Marsh, Reikene, Reiten and Todorov \cite{BMRRT} for general acyclic and finite quivers.

This generalised an earlier category studied by Holm and J\o rgensen \cite{Holm2009}, where the authors look at the finite derived category $\mathrm{D}^{\mathrm{f}}(R)$, for the differential graded algebra $R = k [T]$ with $T$ placed in homological degree 1, and trivial differential.
This is shown to have indecomposable objects in bijection to the set of arcs between non-consecutive integers on the number line.

It has been known since their introduction that a cluster category of Dynkin type $A_n$ has an associated combinatorial model, with indecomposable in bijection to diagonals of the $(n+3)$-gon, and cluster tilting objects in bijection to triangulations of the $(n+3)$-gon.
A discrete cluster category of Dynkin type $A_{\infty}$, labelled $\cC_n$, can be analogously modelled by $(S,\sM)$ a unit circle with infinitely many discrete marked points, $\sM$, on the boundary and $n$ two-sided limits of marked points, known as accumulation points.
The indecomposable objects are in bijection to the diagonals of $(S,\sM)$, and an $\mathrm{Ext}^1$-space between indecomposable objects $X,Y$ is non-zero if and only if the diagonals corresponding to $X,Y$ cross.

The case of $(S,\sM)$ having a single accumulation point is equivalent to the category studied by Holm and J\o rgensen in \cite{Holm2009}.

A natural question to ask given the combinatorial model is what would happen if we allowed the accumulation points to be included as marked points?

The first answer to this question was given by Fisher in \cite{Fisher2014} for the category studied by Holm and J\o rgensen.
This was done by formally including a certain class of homotopy colimits for slices in the Auslander-Reiten quiver of $\cC_1$, and showing that the resulting category had a triangulated structure

More generally, an answer was given by Paquette and Y\i ld\i r\i m \cite{Paquette2020} for $\cC_n$ for $n \geq 1$, where they constructed a new family of triangulated categories $\ocC$, for $n \geq 1$.
Informally, the authors do this by "opening up" each accumulation points into two accumulation points, adding more marked points in between the new accumulation points, and taking a Verdier localisation of the resulting category by a thick subcategory induced by the new marked points.
This is equivalent to $\ocC$ being a Verdier localisation of $\cC_{2n}$, where the combinatorial model $(S,\sM')$ has $2n$ accumulation points $\{a_1,\ldots,a_{2n}\}$ with a cyclic order, by a thick subcategory with indecomposable objects in bijection to the diagonals with both endpoints between the accumulation points $a_{2i}$ and $a_{2i+1}$ for some $1 \leq i \leq n$.

The authors in \cite{Paquette2020} describe the cluster tilting subcategories of $\ocC$, and define a cluster character for all such subcategories.

Our main goal is to study the classical generators of $\ocC$ and, more generally, we wish to determine the thick closure of any given object in $\ocC$.
We also look at the generation times of the classical generators, and subsequently the Orlov spectrum and Rouquier dimension.

An object $X$ in a triangulated category $\cT$ is a classical generator if the smallest thick subcategory containing $X$ is the category $\cT$ itself.
Classical generators have been used; to determine a dimension on triangulated categories \cite{Rouquier}, to find an equivalence between an algebraic triangulated category and the derived category of a differential graded algebra \cite{KellerDeriving}, and as a necessary condition on certain triangulated categories to be saturated \cite{BvdBGenerators}.

In order to do this, we introduce the notion of \textit{homologically connected} objects in a $\mathrm{Hom}$-finite, $k$-linear triangulated category with suspension functor $\wb{1}$.
An object $X$ in a triangulated category $\cT$ is homologically connected if there exists a sequence of non-zero morphisms of degree 1 (i.e.\ elements of $\Ext{1}{-}{-}$) between any two indecomposable objects in the subcategory,
\[
\lang[1]{X} =\mathrm{add}\{X \wb{i} \mid i \in \bZ\}.
\]
Suppose $X$ is a homologically connected object, then we say that the \textit{homological length} of $X$ is the supremum of all minimal sequences of non-zero morphisms between any two indecomposable objects in $\lang[1]{X}$.

We show that, given a $\mathrm{Hom}$-finite, $k$-linear triangulated category $\cT$ such that all indecomposable objects are homologically connected, then an object $X \in T$ a unique decomposition into homologically connected direct summands $X \cong \bigoplus_{i \in I} X_i$, such that $X_i \oplus X_j$ is not homologically connected for all $i,j \in I$, $i \neq j$.
This decomposition is called the \textit{hc (=homologically connected) decomposition} of $X$.

Thick subcategories of $\cC_n$ were classified by Gratz and Zvonareva in \cite{GZ2021} by means of an isomorphism between the lattice of thick subcategories of $\cC_n$ and the lattice of non-exhaustive non-crossing partitions of $[n]$, labelled $NNC_n$.

We adapt this approach to the thick subcategories of $\ocC$, introducing a related concept of even-exclusive non-exhaustive non-crossing partitions of $[2n]$, and showing in Theorem \ref{Thm:Lattices} that there is an isomorphism between the lattice of thick subcategories of $\ocC$ and the lattice of even-exclusive non-exhaustive non-crossing partitions of $[2n]$, called $eNNC_{2n}$.
Furthermore, we provide a means for determining the thick closure of any object $X \in \ocC$ by associating a partition in $eNNC_{2n}$ to the unique hc decomposition of $X$.

\begin{theorem}[Theorem \ref{Thm:Lattices}]
    There is an isomorphism of lattices,
    \[
    eNNC_{2n} \cong \mathrm{thick}(\ocC).
    \]
    Moreover, there is a commutative diagram of lattices 
    \[
    \begin{tikzcd}
        NNC_{2n} \arrow[d,"\eta",swap] \arrow[rr,"\sim"] && \mathrm{thick}(\cC_{2n}) \arrow[d,"\pi",swap]\\
        eNNC_{2n}  \arrow[rr,"\sim"] && \mathrm{thick}(\ocC) \\
        NNC_n \arrow[rr,"\sim"] \arrow[u,"\zeta"] && \mathrm{thick}(\cC_n) \arrow[u,"\xi"].
    \end{tikzcd}
    \]
\end{theorem}

The classical generators of $\ocC$ are therefore those objects that have a hc decomposition corresponding to the maximum element in $eNNC_{2n}$.
We provide a means for determining which objects in $\ocC$ have such a hc decomposition, using a maximality condition on objects defined using the combinatorial model.

\begin{theorem}[Theorem \ref{Thm:GensOfCn}]
Let $G$ be an object in $\ocC$, then $G$ is a classical generator of $\ocC$ if and only if $G$ is homologically connected and $G$ has a complete orbit of $\mathscr{M}$.
\end{theorem}

It is shown that the homological length of a classical generator provides an upper bound on the generation time of the classical generator.
We show that for all $n \geq 1$, there exists a classical generator that has homological length 1, and use this to show that $\ocC$ has a Rouquier dimension of 1.

Finally, we show that a classical generator of $\ocC$ must have a homological length of at most $2n-2$, and show that there will always be a classical generator with homological length $2n-2$.
We use this to prove our final main result, that the Orlov spectrum of $\ocC$ is bounded above by $2n-2$.

\begin{theorem}[Theorem \ref{Thm:Orlov}]
The Orlov spectrum of $\ocC$ for $n \geq 2$ is bounded above by $2n-2$.
That is
\[
\cO(\ocC) \subseteq \{1,\ldots,2n-2\}.
\]
\end{theorem}

Section \ref{Sec:Prelims} is an overview of Paquette-Y\i ld\i r\i m categories and classical generators of triangulated categories.
In Section \ref{Sec:HomCon} we introduce the notion of homologically connected objects, and study them in $\ocC$.
Section \ref{Sec:Thick} is where we classify the thick subcategories of $\ocC$ and construct an isomorphism of lattices to the lattice of thick subcategories.
Finally, in Section \ref{Sec:GensofCn} we determine the classical generators of $\ocC$, and compute the Rouquier dimension and provide an upper bound for the Orlov spectrum.

This work has significant overlap with the author's PhD thesis, in particular Sections \ref{Sec:HomCon} and \ref{Sec:GensofCn}.

\subsection*{Acknowledgements}

The author would like to thank Sira Gratz and Greg Stevenson for their support and advice, both in this work and the overlapping work in the author's PhD thesis.
Thanks also goes to David Pauksztello and Sofia Franchini for their interesting discussions that prompted the idea for this project. 

%% file: Sections/Prelims.tex
\section{Preliminaries}\label{Sec:Prelims}

\subsection{Paquette-Y\i ld\i r\i m Categories}

Let $S$ denote a unit circle with an anticlockwise orientation. Throughout we let $k$ be a field.

\begin{definition}\cite{Gratz2017}\label{Def:admissible}
    A subset $\sM \subset S$ is called \textit{admissible} if the following hold;
    \begin{itemize}
        \item $\sM$ is discrete, that is, there exists an open neighbourhood in $S$ for each element $x \in \sM$ containing no other element of $\sM$,
        \item $\sM$ has infinitely many elements,
        \item $\sM$ satisfies the two-sided limit condition, that is, every point $a \in S$ that is the limit of a sequence in $\sM$ is the limit of both an increasing and decreasing sequence in $\sM$ with respect to the cyclic order.
    \end{itemize}
\end{definition}

We say that any point $a \in S$ that is the limit of both an increasing and decreasing infinite sequence is an \textit{accumulation point}, and label the set of accumulation points by $L(\sM)$.
Note that by the discrete nature of $\sM$, then $\sM \cap L(\sM) = \emptyset$, and we say that $\osM := \sM \cup L(\sM)$.

We impose an ordering on $\osM$ by saying that for $x,y,z \in \osM$ we have
\[
x < y < z,
\]
if and only if when starting at $x$ and going in an anticlockwise direction, we first reach $y$ and then $z$.
It follows from Definition \ref{Def:admissible} that each element $x \in \sM$ has a unique \textit{predecessor} $x^- \in \sM$ and a unique \textit{successor} $x^+ \in \sM$, such that $(x^-,x) \cap \osM= \emptyset = (x,x^+) \cap \osM$.
Here $(u,v) \subset \osM$ denotes the subset of marked points $\{z \in \osM\}$ such that $u < z < v$.

The set $\osM$ is no longer discrete, as an open neighbourhood around $a \in L(\sM)$ will always contain an element of $\sM$, and so $a$ does not have a unique successor or predecessor.
To this, when discussing successors and predecessors in $\osM$, we fix $a^- = a^+ = a \in L(\sM) \subset \osM$.

\begin{definition}
    An \textit{arc} $\ell = \{x_1,x_2\}$ \textit{of} $\sM$ (resp.\ $\osM$) consists of a pair of distinct elements $x_1,x_2 \in \sM$ (resp.\ $x_1,x_2 \in \osM$) such that $x_2 \neq x_1,x_1^-,x_1^+$.
    We say that a pair of arcs $\ell = \{x_1,x_2\}$ and $\ell' = \{y_1,y_2\}$ \textit{cross} if 
    \[
    x_1 < y_1 < x_2 < y_2 < x_1 \quad \text{or} \quad x_1 < y_2 < x_2 < y_1 < x_1.
    \]
    The set of arcs of $\sM$ (resp.\ $\osM$) is denoted $\mathrm{arc}(\sM)$ (resp.\ $\mathrm{arc}(\osM)$).
\end{definition}

It follows from $\sM$ being a subset of $\osM$ that every arc of $\sM$ is also an arc of $\osM$.

We say that two arcs $\ell$ and $\ell'$, \textit{share an endpoint} if $\ell = \{x,y\}$ and $\ell' =\{x,z\}$.
There exists a map $\wb{1} : \mathrm{arc}(\sM) \rightarrow \mathrm{arc}(\sM)$ (resp.\ $\wb{1} : \mathrm{arc}(\osM) \rightarrow \mathrm{arc}(\osM)$) that takes an arc $\ell = \{x_1,x_2\}$ to its clockwise rotation, $\ell \wb{1} = \{ x_1^-,x_2^-\}$.
The map $\wb{1}$ has an inverse map $\wb{-1} : \mathrm{arc}(\sM) \rightarrow \mathrm{arc}(\sM)$ (resp.\ $\wb{-1} : \mathrm{arc}(\osM) \rightarrow \mathrm{arc}(\osM)$) that takes an arc to its anticlockwise rotation, $\ell \wb{-1} = \{ x_1^+,x_2^+\}$.

Let $\ell=\{x_1,x_2\}$ be an arc of $\sM$, then $\ell$ is one of the following types of arc;
\begin{itemize}
    \item $\ell$ is a \textit{short arc} if $x_1,x_2 \in \sM$ and $a < x_1 < x_2^- < a'$ for $a,a' \in L(\sM)$ such that $a < b < a'$ for $b \in L(\sM)$ implies $b=a$ or $b=a'$,
    \item $\ell$ is a \textit{long arc} if $x_1,x_2 \in \sM$ and $\ell$ is not a short arc.
\end{itemize}

Further, if $\ell'=\{x'_1,x'_2\}$ is an arc of $\osM$, then $\ell'$ is either a short arc, a long arc, or one of the two following types of arcs;
\begin{itemize}
    \item $\ell'$ is a \textit{limit arc} if $x'_1 \in \sM$ and $x'_2 \in L(\sM)$, up to relabelling,
    \item $\ell'$ is a \textit{double limit arc} if $x'_1,x'_2 \in L(\sM)$.
\end{itemize}

It is straightforward to verify that if $\ell$ is a double limit arc, then $\ell \wb{1} = \ell \wb{-1} = \ell$.
We note that in a case with a single accumulation point, all arcs are either short arcs or limit arcs.
At times, we refer to (double) limit arcs, this is taken to mean the collection of all limit arcs and double limit arcs.

\begin{figure}[h!]
    \centering
    \begin{tikzpicture}
        \draw (0,0) circle (3cm);
        \draw[thin] ({3*sin(315)},{3*cos(315)}) -- ({3*sin(135)},{3*cos(135)}) node[pos=0.5, above] {$\ell_4$};
        \draw[thin] ({3*sin(315)},{3*cos(315)}) .. controls(-0.8,-0.3).. ({3*sin(190)},{3*cos(190)}) node[pos=0.6, left] {$\ell_3$};
        \draw[thin] ({3*sin(258)},{3*cos(258)}) .. controls(-2.5,0.2).. ({3*sin(293)},{3*cos(293)}) node[pos=0.5, right] {$\ell_1$};
        \draw[thin] ({3*sin(23)},{3*cos(23)}) .. controls(1.6,0.8) .. ({3*sin(108)},{3*cos(108)}) node[pos=0.6, left] {$\ell_2$};
        \draw[fill=white] ({3*sin(45)},{3*cos(45)}) circle (0.1cm);
        \draw[fill=white] ({3*sin(135)},{3*cos(135)}) circle (0.1cm);
        \draw[fill=white] ({3*sin(225)},{3*cos(225)}) circle (0.1cm);
        \draw[fill=white] ({3*sin(315)},{3*cos(315)}) circle (0.1cm);
        \node at ({3.3*sin(315)},{3.3*cos(315)}) {\scriptsize $a$};
        \node at ({3.3*sin(135)},{3.3*cos(135)}) {\scriptsize $a'$};
        \node at ({3.2*sin(293)},{3.2*cos(293)}) {\scriptsize $x_1$};
        \node at ({3.2*sin(258)},{3.2*cos(258)}) {\scriptsize $x_2$};
        \node at ({3.2*sin(190)},{3.2*cos(190)}) {\scriptsize $x_3$};
        \node at ({3.2*sin(108)},{3.2*cos(108)}) {\scriptsize $x_4$};
        \node at ({3.2*sin(23)},{3.2*cos(23)}) {\scriptsize $x_5$};
    \end{tikzpicture}
    \caption{Examples of the four classes of arcs of $\osM$, where $\ell_1$ is a short arc, $\ell_2$ is a long arc, $\ell_3$ is a limit arc, and $\ell_4$ is a double limit arc. The small circles on the boundary represent the accumulation points in $\osM$.}
    \label{fig:arcs}
\end{figure}
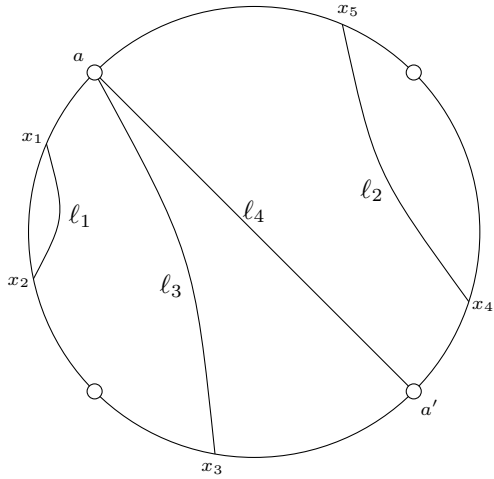

Given some admissible subset $\sM \subset S$, Igusa and Todorov \cite{Igusa2013} construct a cluster category $\cC(S,\sM)$ such that indecomposable objects are in bijection to the arcs of $\sM$, where an indecomposable object $X$ corresponds to an arc $\ell_X$.
For two indecomposables objects $X,Y \in \cC(S,\sM)$, 
\[
\Hom{X}{Y} \cong \begin{cases*}
    k & if $\ell_X, \ell_{Y\wb{-1}}$ cross,\\
    0 & else.
\end{cases*}
\]
The cluster category $\cC(S,\sM)$ is a 2-Calabi-Yau, $\mathrm{Hom}$-finite, $k$-linear, Krull-Schmidt triangulated category, with suspension functor $\wb{1}$, given by taking an indecomposable object $X$ to $X\wb{1}$, the indecomposable object that corresponds to the arc $\ell_{X\wb{1}}$.
We call $\cC(S,\sM)$ a \textit{discrete cluster category of Dynkin type} $A_{\infty}$ \textit{with respect to} $(S,\sM)$.

Similarly, we construct a cluster category $\mathcal{C}(S,\osM)$ from a marked surface $(S,\osM)$, where indecomposable objects are in bijection with the arcs of $\osM$.
Given two indecomposable objects $X,Y \in \mathcal{C}(S,\osM)$, then $\Hom{X}{Y} \cong k$ if and only if $\ell_X$ and $\ell_Y \wb{-1}$ cross, or $\ell_X$ and $\ell_Y \wb{-1}$ are both (double) limit arcs sharing an endpoint at an accumulation point, and $\ell_Y \wb{-1}$ is an anti-clockwise rotation of $\ell_X$ about their shared endpoint, else $\Hom{X}{Y}=0$.

For any two indecomposable objects $X$ and $Y$ in $\cC(S,\sM)$ (resp.\ $\cC(S,\osM)$) such that $\ell_X =\{x_0,x_1\}$ and $\ell_Y=\{y_0,y_1\}$ cross, then 
\begin{align*}
    X \rightarrow A \oplus B \rightarrow Y \rightarrow X \wb{1},\\
    Y \rightarrow C \oplus D \rightarrow X \rightarrow Y \wb{1},
\end{align*}
are both distinguished triangles in $\cC(S,\sM)$ (resp.\ $\cC(S,\osM)$), such that $\ell_A = \{x_0,y_1\}$, $\ell_B = \{y_0,x_1\}$, $\ell_C = \{x_0,y_0\}$ and $\ell_D=\{x_1,y_1\}$.

If $\ell_X = \{x,x_1\}$ and $\ell_Y = \{x,y_1\}$ are arcs with $x \in L(\sM)$, and $\ell_Y$ is an anticlockwise rotation of $\ell_X$, then 
\[
Y \rightarrow Z \rightarrow X \rightarrow Y \wb{1} 
\]
is a distinguished triangle, with $\ell_Z = \{x_1,y_1\}$.

\begin{figure}[H]
    \centering
    \begin{tikzpicture}
    \draw (0,0) circle (2cm);
    \draw[thin] ({2*sin(37)},{2*cos(37)}) .. controls(0,0) .. ({2*sin(196)},{2*cos(196)});
    \draw[thin] ({2*sin(281)},{2*cos(281)}) .. controls(0,0) .. ({2*sin(98)},{2*cos(98)});
    \draw[dotted, thin] ({2*sin(37)},{2*cos(37)}) ..controls(1.3,0.53).. ({2*sin(98)},{2*cos(98)});
    \draw[dotted, thin] ({2*sin(196)},{2*cos(196)}) ..controls(-1.2,-0.6).. ({2*sin(281)},{2*cos(281)});
    \draw[dashed, thin] ({2*sin(37)},{2*cos(37)}) ..controls(-0.2,0.8).. ({2*sin(281)},{2*cos(281)});
    \draw[dashed, thin] ({2*sin(196)},{2*cos(196)}) ..controls(0.5,-1.1).. ({2*sin(98)},{2*cos(98)});
    \node at ({2.3*sin(196)},{2.3*cos(196)}) {\footnotesize $x_1$};
    \node at ({2.3*sin(37)},{2.3*cos(37)}) {\footnotesize $x_0$};
    \node at ({2.3*sin(281)},{2.3*cos(281)}) {\footnotesize $y_0$};
    \node at ({2.3*sin(98)},{2.3*cos(98)}) {\footnotesize $y_1$};
    \node at (0.1,-0.8) {\footnotesize $\ell_X$};
    \node at (0.9,0.1) {\footnotesize $\ell_Y$};
    \node at (-0.3,1.2) {\footnotesize $\ell_C$};
    \node at (0.8,-1.2) {\footnotesize $\ell_D$};
    \node at (1.7,0.5) {\footnotesize $\ell_A$};
    \node at (-1.4,-0.8) {\footnotesize $\ell_B$};
    \end{tikzpicture}
    \caption{Arcs corresponding to the distinguished triangles induced by non-trivial $\Ext{1}{X}{Y} \cong \Ext{1}{Y}{X}$ groups.}
    \label{fig:ExTri}
\end{figure}
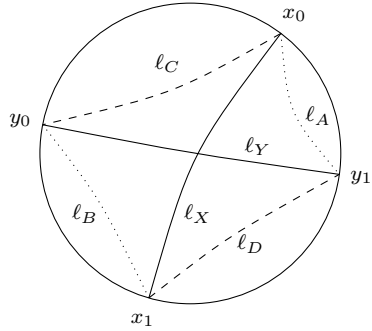

Furthermore, let $f:X \rightarrow Y$ be a morphism between indecomposable objects in either $\cC(S,\sM)$ or $\cC(S,\osM)$, where $\ell_X=\{x_0,x_1\}$ and $\ell_Y = \{y_0,y_1\}$, then $f$ factors through an indecomposable object $S$ if and only if
\[
x_0 \leq s_0 \leq y_0 < x_1, \, \text{and} \,\, x_1 \leq s_1 \leq y_1 < x_0,
\]
where $\ell_S = \{s_0,s_1\}$.

We call a category $\cC(S,\osM)$ the \textit{Paquette-Y\i ld\i r\i m completion of the discrete cluster category of Dynkin type} $A_{\infty}$ \textit{with respect to} $(S,\osM)$, or the Paquette-Y\i ld\i r\i m category of $(S,\osM)$ for short.
The Paquette-Y\i ld\i r\i m category of $(S,\osM)$ is a Krull-Schmidt, $\mathrm{Hom}$-finite, $k$-linear, triangulated category with suspension functor $\wb{1}$.
We identify the arc $\ell_{X \wb{1}}$, corresponding to the object $X \wb{1}$, with the arc $\ell_X \wb{1}$.

Given two admissible subsets of $S$, $\sM$ and $\mathscr{N}$, with an equal number of accumulation points, there exists an equivalence of triangulated categories $\cC(S,\osM) \xrightarrow{\sim}  \cC(S,\mathscr{N}^{cl})$, and so we may talk about a Paquette-Y\i ld\i r\i m category up to the number of accumulation points.
Therefore for all $n \geq 1$, we fix some admissible subset $\osM[n] \subset S$ with $n$ accumulation points, and say that $\cC(S,\osM[n]) = \ocC$.
The above equivalence also holds for $\cC(S,\sM)$ and $\cC(S,\mathscr{N})$, so we refer to the category $\cC(S,\sM_n) = \cC_n$.

Paquette-Y\i ld\i r\i m categories were introduced by Paquette and Y\i ld\i r\i m as a Verdier localisation of $\cC_{2n}$ in \cite{Paquette2020}.
It was shown by August, Cheung, Faber, Gratz, and Schroll in \cite{ACFGS} that $\ocC[1]$ is equivalent to the category of $\mathbb{Z}$-graded maximal Cohen-Macauley modules over $\mathbb{C}[x,y]/(x^2)$, with $x$ in degree 1 and $y$ in degree $-1$.

\begin{definition}\label{Def:Orbits}
Let $X$ be an indecomposable object in $\cC_n$ (resp.\ $\ocC$), and let $\sM_X \subseteq \sM$ (resp.\ $\osM[X] \subseteq \osM$) be the set of marked points that are the endpoints of the arcs corresponding to suspensions and desuspensions of $X$.

If $A \cong \bigoplus^l_{i=1} X_i$, with all $X_i$ indecomposable, then $\sM_A = \bigcup_{i=1}^l \sM_{X_i}$.
We call $\sM_A$ the \textit{orbit of} $A$ \textit{in} $\sM$ (resp.\ \textit{orbit of} $A$ \textit{in} $\osM$).
If we have $\sM_A=\sM$ (resp.\ $\osM[A] = \osM$), then we say $A$ has a \textit{complete orbit in} $\sM$ (resp.\ $A$ has a \textit{complete orbit in} $\osM$).
\end{definition}

There are times, particularly when classifying thick subcategories, where it will be useful to think of $\ocC$ as a Verdier localisation of $\cC_{2n}$ rather than the direct combinatorial model described above.
Therefore we briefly go over that construction here, see \cite{Paquette2020} for full details.

\begin{construction}\label{Cons:Verdier Localisation}
Let $\cC_{2n}$ be the discrete cluster category of Dynkin type $A_{\infty}$ with $2n$ accumulation points, labelled $a_1,\ldots,a_{2n}$ with cyclic ordering $a_i< a_{i+1} < a_i$.
Let $\cD \subset \cC_{2n}$ be the subcategory defined by taking short arcs with endpoints in $\bigcup_{i=1}^{n} (a_{2i},a_{2i+1})$, and let $\Omega$ be the set of morphisms in $\cC_{2n}$ such that $f \in \Omega$ if and only if $\mathrm{cone}(f) \in \cD$.
The subcategory $\cD$ is a thick subcategory of $\cC_{2n}$ \cite{Paquette2020}, and is equivalent to the disjoint union of $n$ copies of $\cC_1$ \cite{Murphy22}.

Then
\[
\pi \colon \cC_{2n}[\Omega^{-1}] \xrightarrow{\sim} \ocC,
\]
that is, $\ocC$ is equivalent to the Verdier localisation of $\cC_{2n}$ with respect to the subcategory $\cD$.
\end{construction}

\begin{figure}[H]
\centering
    \begin{tikzpicture}
        \draw[fill=black, opacity=0.2] (-3.5,0) circle (2.5cm);
        \path[fill=white] (0,-1) -- (-2.5,-3.5) -- (-6,0) -- (-6,3) -- (-4.5,3.5) -- cycle;
        \draw (-3.5,0) circle (2.5cm);
        \draw (3.5,0) circle (2.5cm);
        \draw[fill=white] ({-3.5+2.5*sin(0)},{2.5*cos(0)}) circle (0.1cm);
        \draw[fill=white] ({-3.5+2.5*sin(90)},{2.5*cos(90)}) circle (0.1cm);
        \draw[fill=white] ({-3.5+2.5*sin(180)},{2.5*cos(180)}) circle (0.1cm);
        \draw[fill=white] ({-3.5+2.5*sin(270)},{2.5*cos(270)}) circle (0.1cm);
        \draw[fill=white] ({3.5+2.5*sin(45)},{2.5*cos(45)}) circle (0.1cm);
        \draw[fill=white] ({3.5+2.5*sin(225)},{2.5*cos(225)}) circle (0.1cm);
        \draw[very thin] ({-3.5+2.5*sin(45)},{2.5*cos(45)}) -- ({-3.5+2.5*sin(225)},{2.5*cos(225)});
        \draw[very thin] ({-3.5+2.5*sin(158)},{2.5*cos(158)}) -- ({-3.5+2.5*sin(338)},{2.5*cos(338)});
        \draw[very thin] ({-3.5+2.5*sin(294)},{2.5*cos(294)}) .. controls(-3.8,1.5) .. ({-3.5+2.5*sin(35)},{2.5*cos(35)});
        \draw[very thin] ({-3.5+2.5*sin(190)},{2.5*cos(190)}) .. controls (-4.3,-2) .. ({-3.5+2.5*sin(218)},{2.5*cos(218)});
        \draw[very thin] ({-3.5+2.5*sin(110)},{2.5*cos(110)}) .. controls (-1.8,-1.3) .. ({-3.5+2.5*sin(145)},{2.5*cos(145)});
        \draw[very thin] ({3.5+2.4*sin(45)},{2.4*cos(45)}) -- ({3.5+2.4*sin(225)},{2.4*cos(225)});
        \draw[very thin] ({3.5+2.5*sin(158)},{2.5*cos(158)}) -- ({3.5+2.5*sin(338)},{2.5*cos(338)});
        \draw[very thin] ({3.5+2.5*sin(110)},{2.5*cos(110)}) .. controls (5.2,-1.3) .. ({3.5+2.5*sin(145)},{2.5*cos(145)});
        \draw[very thin] ({3.5+2.5*sin(294)},{2.5*cos(294)}) .. controls(3.2,1.3) .. ({3.5+2.4*sin(45)},{2.4*cos(45)});
        \draw[->>] (-0.75,0) -- (0.75,0) node[pos=0.5, above] {$\pi$};
    \end{tikzpicture}
    \caption{A combinatorial model of the localisation $\pi:\cC_{2n} \rightarrow \ocC$, where an indecomposable object is in $\cD$ if and only if the corresponding arc is entirely in one of the shaded areas, up to isotopy.}
    \label{fig:enter-label}
\end{figure}
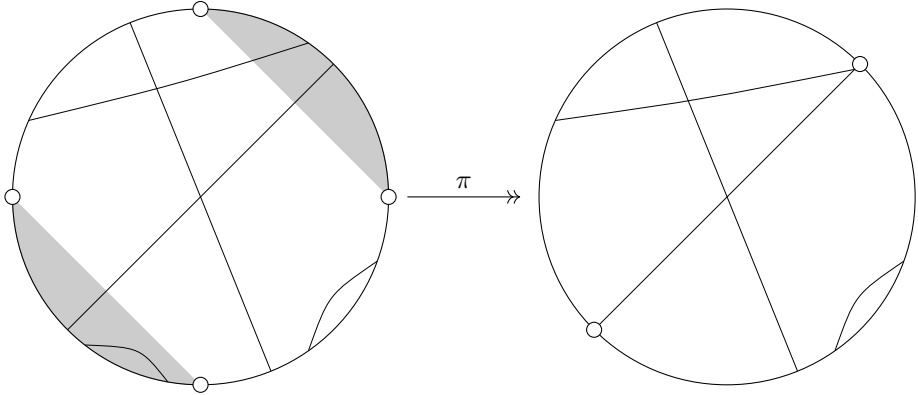

\subsection{Classical Generators}\label{Subsec:ClassGens}

Objects that generate a triangulated category have been the subject of study for many years, notably being used to construct equivalences of categories by Keller \cite{KellerDeriving} for algebraic triangulated categories.
Here we recall some definitions about classical generators, many of these are adapted versions of the definitions given in \cite{BvdBGenerators} by Bondal and Van den Bergh.

\begin{definition}\label{Def:thick subcat}
Let $\cT$ be a triangulated category with a triangulated subcategory $\cB$, then we say $\cB$ is a \textit{thick subcategory} of $\cT$ if it is closed under direct summands.
\end{definition}

\begin{definition}\label{Def:Gen}
Let $G$ be an object in $\cT$, then we say that $E$ \textit{classically generates} $\cT$ if the smallest thick subcategory of $\cT$ containing $G$ is $\cT$ itself.
Such an object $G$ is called a \textit{classical generator}.
\end{definition}

The term \textit{classical generator} is used to distinguish between this and other notions of generators, for instance a \textit{(weak) generator} $G$ in a triangulated category is an object such that for all objects $K$, there exists a non-zero morphism $G \rightarrow K \wlb n \wrb$ for some $n \in \mathbb{Z}$.
It should be noted that all classical generators are weak generators, however the converse is not always true.
Whenever we refer to generators, we shall always be referring to classical generators.

\begin{definition}\label{Def:StarProd}
Let $\mathcal{R,S} \subset \cT$ be two subcategories, then $\cR \star \cS$ is the full subcategory of direct summands of all objects $Y$ such that there exists a triangle 
\[
X \rightarrow Y \rightarrow Z \rightarrow X\wlb 1 \wrb
\]
with $X \in \cR$ and $Z \in \cS$.
\end{definition}

It can be directly verified by the Octahedral axiom that the operation $\star$ is associative.
 
Given some subcategory $\cR \subseteq \cT$, it is possible to build an iterative series of subcategories from $\cR$ using the operation $\star$.
These subcategories prove to be a useful tool in determining whether or not a collection of objects classically generate $\cT$ or not.

\begin{definition}\label{Def:langle}
Let $G$ be an object in $\cT$.
We denote by $\langle G \rangle_1 \subset \cT$ the full subcategory consisting of all direct summands of finite coproducts of suspensions of objects in $G$.
Further, we define the full subcategory $\langle G \rangle_{n+1}$ to be
\[
\langle G \rangle_{n+1} := \langle G \rangle_n \star \langle G \rangle_1.
\]
Moreover, we define $\langle G \rangle$ to be the union of all $\langle G \rangle_n$, i.e.\
\[
\langle G \rangle := \; \bigcup_n \; \langle G \rangle_n. 
\]
\end{definition}

Associativity of the operation $\star$ may be used to show that $\lang[a]{G} \star \lang[b]{G} = \lang[c]{G}$ such that $a+b = c$.

For an object $G \in \cT$, the subcategory $\lang{G}$ is the smallest thick subcategory of $\cT$ that contains $G$, that is $G$ is a classical generator if and only if $\lang{G}=\cT$.
Further, if there exists $n \in \mathbb{Z}$, such that $\langle G \rangle_n = \cT$, then we call $G$ a \textit{strong generator of} $\cT$ and say that $G$ generates $\cT$ in $n$ steps.
If some $G \in \cT$ is a strong generator, then all classical generators of $\cT$ are also strong generators.

We shall use the following definition throughout the next few sections, as it is an important requirement used in the proofs of some of the results stated.

\begin{definition}
Let $G=\bigoplus_{j \in J} G_j$ be a classical (resp.\ strong) generator of some triangulated category $\cT$.
We say that $G$ is a \textit{minimal classical (resp.\ strong) generator} of $\cT$ if there exists no classical (resp.\ strong) generator, $G'$, such that $G \cong G' \oplus G_j$ for some $j \in J$. 
\end{definition}

If a minimal strong generator $E \in \cT$ generates $\cT$ in $n$ steps, then a strong generator $E' \cong E \oplus F$, for some $F \in \cT$, must generate $\cT$ in $\leq n$ steps, as $\langle E \rangle_i \subseteq \langle E' \rangle_i$ for all $i \geq 0$, and $\langle E \rangle_n = \cT$.

\section{Homologically Connected Objects}\label{Sec:HomCon}

We say that a morphism $f \in \Ext{i}{X}{Y}$ is a \textit{morphism of degree} $i$.
Let $f \in \Ext{1}{X}{Y}$ be a morphism of degree 1, then we say that an object $Z$ in a distinguished triangle 
\[
Y \rightarrow Z \rightarrow X \xrightarrow{f} Y\wlb 1 \wrb,
\]
is an \textit{extension of} $X$ by $Y$.

\begin{definition}
Let $\cT$ be a $\mathrm{Hom}$-finite, Krull-Schmidt triangulated category.
Let $G= \bigoplus_{i=1}^m G_i$ be an object in $\cT$.
Then we say $G$ is \textit{homologically connected} if for any two indecomposable objects $F_1$ and $F_{l+1}$ in $\langle G \rangle_1$, then there is some finite set of non-zero morphisms of degree 1 with a non-trivial extension, $f_1,\ldots,f_l$, between indecomposable objects in $\langle G \rangle_1$ that form a sequence between $F_1$ and $F_{l+1}$;
\[
\begin{tikzcd}
F_1 \arrow[r,dash,"f_1"] & F_2 \arrow[r,dash,dashed] & F_l \arrow[r,dash,"f_l"] & F_{l+1}.
\end{tikzcd}
\]
We call these sequences a \textit{zig-zag from} $F_1$ \textit{to} $F_{l+1}$.
\end{definition}

The direction of the morphisms and their composition is not required, it is sufficient to know that there is a zig-zag between all indecomposable direct summands of $G$.
We say that a zig-zag between two indecomposable objects has \textit{length} $l$ if there are $l$ morphisms of degree 1 in the zig-zag, and a zig-zag in \textit{minimal} if it has the smallest length of all zig-zags between the same objects.
We fix it so that a minimal zig-zag between two isomorphic objects has length zero.
The \textit{homological length} of a homologically connected object $G$ is the supremum of the length of all minimal zig-zags between any two indecomposable direct summands of $G$.
If no supremum exists, then we say that the homological length is $\infty$.

\begin{lemma}\label{Lem:decomposition}
Let $\cT$ be a $\mathrm{Hom}$-finite, Krull-Schmidt triangulated category such that each indecomposable object is homologically connected, and let $G$ be an object in $\cT$.
    Then there exists a decomposition of $G \cong \oplus_{i=1}^l X_i$, such that $X_a$ is homologically connected for all $a=1, \ldots, l$, and $X_a \oplus X_b$ is not homologically connected for $a \neq b$.

    Moreover, this decomposition is unique up to relabelling.
\end{lemma}

\begin{proof}
    Suppose that $G \cong \bigoplus_{i=1}^m G_i$ where $G_i$ are all indecomposable.
    We construct a graph $Q$ with $m$ vertices, with an edge between $i$ and $j$ if $G_i \oplus G_j$ is homologically connected.
    Then we either have a set of disjoint graphs, or the graph is connected.
    We label the connected subgraphs $Q_1,\ldots,Q_l$.
    
    For a connected subgraph $Q_a$ with vertices $\{a_1, \ldots, a_p\}$, let $X_a \cong \bigoplus_{i=1}^p G_{a_i}$.
    Then $X_a$ is homologically connected as each indecomposable direct summand is homologically connected by assumption, and an edge between $a_i$ and $a_j$ represents a zig-zag between $G_{a_i}$ and $G_{a_j}$ in $\langle X_a \rangle_1$.
    Hence there exists a zig-zag between any two indecomposable objects in $\langle X_a \rangle_1$, meaning $X_a$ is homologically connected.

    As there are no morphisms of degree 1 between the indecomposable objects corresponding to vertices in disjoint subgraphs, then each $X_a$ and $X_b$, for $a \neq b$, have no morphisms of degree 1 between them, and so $X_a \oplus X_b$ is not homologically connected.
    Therefore we have
    \[
    G \cong X_1 \oplus \cdots \oplus X_l,
    \]
    such that each $X_a$ is homologically connected and $X_a \oplus X_b$ is not homologically connected for any $a \neq b$.

    Uniqueness follows from $\cT$ being Krull-Schmidt, and the construction of the graph $Q$.
\end{proof}

We call such a decomposition of an object $X$ the \textit{hc (=homologically connected) decomposition of} $X$.

\begin{remark}
    Any $\mathrm{Hom}$-finite, Krull-Schmidt, 2-Calabi-Yau triangulated category $\cT$, such that $\wb{2}$ is not isomorphic to the identity on indecomposable objects, satisfies the requirement that each indecomposable object is homologically connected.
    This is as there will always be a morphism of degree 1, $f \in \Ext{1}{X}{X\wb{1}}$, that has a non-trivial extension for any indecomposable object $X \in \cT$.
\end{remark}

\subsection{Homologically Connected Objects in \texorpdfstring{$\ocC$}{Cn}}\label{Subsec:HomCon}

The following proposition shows us how to reduce the length of a zig-zag by using the extensions of morphisms of degree 1 in the zig-zag.

\begin{lemma}\label{Lem:ZigZagReduction}
Let $X,Y,Z \in \ocC$ be indecomposable objects such that
\[
\begin{tikzcd}
X \arrow[r,dash] & Y \arrow[r,dash] & Z
\end{tikzcd}
\]
is a minimal zig-zag between $X,Z \in \langle G \rangle$, for a homologically connected object $G \in \ocC$.
Let $X \in \langle G \rangle_a$ and $Y \in \langle G \rangle_b$, such that $Y \not\cong X \wlb \pm 1 \wrb$ if $\ell_X$ is a limit arc.
Then there exists an indecomposable object $A \in \langle G \rangle_{a+b}$ that is a direct summand of an extension of $X$ by $Y$, or a direct summand of an extension of $Y$ by $X$, such that there exists a minimal zig-zag
\[
\begin{tikzcd}
    A \arrow[r,dash] & Z.
\end{tikzcd}
\]
If $\ell_X$ is a limit arc and $Y \cong X \wlb 1 \wrb$ (resp.\ $Y \cong X \wlb -1 \wrb$), then such an $A$ exists as a direct summand of an extension of $Y \wlb 1 \wrb$ by $X$ (resp.\ $A$ is a direct summand of an extension of $X$ by $Y \wlb -1 \wrb$). 
\end{lemma}

\begin{proof}
We have two cases to consider, where $\ell_X$ and $\ell_Y$ cross, and where they share an endpoint.
First, we consider when $\ell_X$ and $\ell_Y$ cross.

Let $\ell_X =\{x_1,x_2\},\ell_Y =\{y_1,y_2\}$ and $\ell_Z=\{z_1,z_2\}$, where $x_1 < y_1 < x_2 < y_2 < x_1$.
Let $f \in \Ext{1}{X}{Y}$, then there exists a triangle
\[
\begin{tikzcd}
Y \arrow[r] & A \oplus A' \arrow[r] & X \arrow[r] & Y\wlb 1 \wrb
\end{tikzcd}
\]
where $\ell_{A \oplus A'} = \ell_A \oplus \ell_{A'} =\{x_1,y_1\} \oplus \{x_2,y_2\}$, and $A,A' \in \langle G \rangle_{a+b}$ as $X \in \langle G \rangle_a$ and $Y \in \langle G \rangle_b$.
As there is a non-zero morphism of degree 1 between $Y$ and $Z$ in at least one direction, then either $\ell_Y$ and $\ell_Z$ cross, or they share an endpoint.
If $\ell_Z$ and $\ell_Y$ share an endpoint, then $\ell_Z$ and $\ell_A$ share an endpoint (or $\ell_Z$ and $\ell_{A'}$ share an endpoint), and therefore there is a morphism of degree 1 in some direction between $Z$ and $A$ (or between $Z$ and $A'$).

Suppose $\ell_Z$ and $\ell_Y$ cross, and further suppose $y_1 < z_1 < y_2 < z_2 < y_1$.
However, we know $\ell_X$ and $\ell_Z$ cannot cross or share an endpoint as there are no morphisms of degree 1 between $X$ and $Z$ by minimality of the zig-zag.
Therefore one of the following holds
\begin{align*}
    y_1 < x_2 \leq z_1 < y_2 < z_2 \leq x_1 < y_1,\\
     y_1 < z_2 \leq x_2 < y_2 < x_1 \leq z_1 < y_1,
\end{align*}
and so $\ell_{A'}$ and $\ell_Z$ cross and thus there is a morphism of degree 1 between $A'$ and $Z$.

Now suppose $\ell_X = \{x,x_1\}$ and $\ell_Y = \{x,y_1\}$ share an endpoint at an accumulation point.
We consider two cases, one where $Y \not\cong X \wlb \pm 1 \wrb$, and the second case where $Y \cong X \wlb 1 \wrb$ (or, $Y \cong X \wlb -1 \wrb$).
For the first case, let $x < x_1 < y_1^- < x$, so we get the triangle 
\[
X \rightarrow A \rightarrow Y \rightarrow X\wlb 1 \wrb
\]
where $\ell_A=\{x_1,y_1\}$.
Let $\ell_Z = \{z_1,z_2\}$, as there is a morphism of degree 1 between $Y$ and $Z$, but no morphisms of degree 1 between $X$ and $Z$, then either $x_1 < z_1 < y_1 < z_2 < x < x_1$, or one of the following holds and $y_1$ is an accumulation point,
\begin{align*}
    x_1 < z_1 <y_1 = z_2 < x < x_1,\\
     x_1 < y_1= z_1 < z_2 < x < x_1.
\end{align*}
Therefore $\ell_A$ either shares an endpoint with $\ell_Z$ at an accumulation point, or they cross, and so there is a morphism of degree 1 between $A$ and $Z$.
Further, if $X \in \langle G \rangle_a$ and $Y \in \langle G \rangle_b$, then $A \in \langle G \rangle_{a+b}$.
A similar argument holds when we consider $x < y_1 < x_1^- < x$.

Next we focus on the second case where $Y \cong X \wlb 1 \wrb$ (resp.\ $Y \cong X \wlb -1 \wrb$).
There exists a non-zero morphism of degree 1 between $X$ and $Y \wlb 1 \wrb$ (resp.\ between $X$ and $Y \wlb -1 \wrb$) by \cite{Paquette2020}.
Any $\ell_Z$ that crosses $\ell_Y$ but not $\ell_X$ either crosses $\ell_{Y \wlb 1 \wrb}$ (resp.\ crosses $\ell_{Y \wlb -1 \wrb}$), or $\ell_Z = \{x_1,x_1^{--}\}$ (resp.\ $\ell_Z = \{x_1,x_1^{++}\}$), and in this case $Z$ is isomorphic to $A$.

If $\ell_Z$ crosses $\ell_{Y \wlb 1 \wrb}$ (resp.\ crosses $\ell_{Y \wlb -1 \wrb}$) we are back in the previous case, and so there exists some object $A \in \langle G \rangle_{a+b}$ such that there exists a zig-zag
\[
\begin{tikzcd}
    A \arrow[r,dash] & Z.
\end{tikzcd}
\]
\end{proof}

The new zig-zags that are formed using Lemma \ref{Lem:ZigZagReduction} are minimal zig-zags, which we show with the next lemma.

\begin{lemma}\label{Lem:Reduce minimal}
    Let $G \in \ocC$ be homologically connected, and let
    \[
\begin{tikzcd}
    G_1 \arrow[r,dash] & G_2 \arrow[r,dash] & \cdots \arrow[r,dash] & G_{d+1},
\end{tikzcd}
\]
be a minimal zig-zag of objects in $\langle G \rangle_1$, with length $d$.
Then there exists a minimal zig-zag
\[
\begin{tikzcd}
    M_i \arrow[r,dash] & G_{i+1} \arrow[r,dash] & \cdots \arrow[r,dash] & G_{d+1},
\end{tikzcd}
\]
of length $d-i+1$, with $M \in \langle G \rangle_{i}$.
\end{lemma}

\begin{proof}
    Let $\ell_{G_i} = \{y_i, z_i\}$ for all $i =1,\ldots,d+1$.
By Lemma \ref{Lem:ZigZagReduction}, there exists an object $M_2 \in \langle G \rangle_2$ such that
\[
\begin{tikzcd}
    M_2 \arrow[r,dash] & G_3 \arrow[r,dash] & \cdots \arrow[r,dash] & G_d \arrow[r,dash] & G_{d+1}
\end{tikzcd}
\]
is a zig-zag.
We repeat this process $d-1$ times using Lemma \ref{Lem:ZigZagReduction}, producing a series of zig-zags for $i = 2,\ldots, d$,
\[
\begin{tikzcd}
    M_i \arrow[r,dash] & G_{i+1} \arrow[r,dash] & \cdots \arrow[r,dash] & G_{d+1},
\end{tikzcd}
\]
where $M_i \in \langle G \rangle_{i-1}$.
We show that this zig-zag is minimal.

If $G_{i+1} \not\cong M_i \wlb \pm 1 \wrb$, or if $M_i$ is a long arc, then we may have $\ell_{M_{i+1}} = \{y_1,z_{i+1}\}$.
This is because $G_{i+1}$ and $M_i$ are indecomposable objects with a morphism of degree 1 between them, and so the cone of the morphism $G_{i+1}\wlb -1 \wrb \rightarrow M_i$ (or $M_i\wlb -1 \wrb \rightarrow G_{i+1}$) is a direct sum of objects corresponding to arcs sharing endpoints with $\ell_{G_{i+1}}$ and $\ell_{M_i}$.

If $M_i$ is a limit arc and $G_{i+1} \cong M_i\wlb 1 \wrb$ (resp.\ $G_{i+1} \cong M_i\wlb -1 \wrb$), then we may replace $G_{i+1}$ in the zig-zag with $G_{i+1}\wlb 1 \wrb \cong M_i \wlb 2 \wrb$ (resp.\ $G_{i+1}\wlb -1 \wrb \cong M_i \wlb -2 \wrb$).
This is possible as $G_{i+2}$ cannot be a short arc by Corollary \ref{Cor:NoShortArcs}, and $\ell_{G_{i+2}}$ crosses $\ell_{G_{i+1}}$ but not $\ell_{M_i}$, so $\ell_{G_{i+2}}$ must cross $\ell_{G_{i+1}\wlb 1 \wrb}$ (resp.\ $\ell_{G_{i+1}\wlb -1 \wrb}$).
In this case we have $\ell_{M_{i+1}} = \{x_1,z_i^-\}$ (resp.\ $\ell_{M_{i+1}} = \{y_1,z_i^+\}$).

As the zig-zag
\[
\begin{tikzcd}
    G_1 \arrow[r,dash] & G_2 \arrow[r,dash] & \cdots \arrow[r,dash] & G_d \arrow[r,dash] & G_{d+1}
\end{tikzcd}
\]
is minimal, then $\ell_{G_a}$ does not cross any arc $\ell_{G_j}$ for $j =a+2,\ldots,d+1$, or share an accumulation as an endpoint.
However, as $\ell_{M_i}$ shares one endpoint with $\ell_{G_1}$ and another with $\ell_{G_i}$, then $\ell_{G_j}$, for $j \geq i+2$ can only cross $\ell_{M_i}$ if it also crosses some $\ell_{G_a}$, or shares an endpoint at an accumulation point with some $\ell_{G_a}$, for $a \leq i$.
This cannot happen, and so there exists no morphisms of degree 1 between $M_i$ and $G_j$ for $j \geq i+2$, and so the zig-zag
\[
\begin{tikzcd}
    M_i \arrow[r,dash] & G_{i+1} \arrow[r,dash] & \cdots \arrow[r,dash] & G_{d+1},
\end{tikzcd}
\]
is minimal with $M_i \in \langle G \rangle_i$.
\end{proof}

Next we show when we can expect an indecomposable object to be homologically connected.

\begin{proposition}\label{Prop:all X hom con}
Let $X \in \ocC$ be an indecomposable object.
Then $X$ is homologically connected.
\end{proposition}

\begin{proof}
We must check the four different types of arc; double limit arcs, limit arcs, long arcs, and short arcs.

If $\ell_X$ is a double limit arc, then $X \cong X[i]$ for all $i \in \mathbb{Z}$, and so $X$ is homologically connected.
Similarly, if $\ell_X$ is a limit arc, then by \cite{Paquette2020} there exists a morphism of degree 1 in $\Ext{1}{X}{X \wb{i}}$ for all $i < 0$, so $X$ is homologically connected.

Now suppose that $\ell_X$ is a long arc, then there exists a morphism of degree 1 in $\Ext{1}{X}{X \wb{i}}$ for all $i \in \mathbb{Z}\backslash \{0\}$, as $\ell_X$ and $\ell_{X[i]}$ cross for all $i \neq 0$.
So $X$ is homologically connected.

Now let $\ell_X$ be a short arc.
Then $\ell_X$ and $\ell_{X\wlb 1 \wrb}$ cross, and so
\[
\Ext{1}{X}{X \wb{1}} \cong k.
\]
Therefore there exists a non-zero morphism of degree 1 in $\Ext{1}{X}{X \wb{1}}$, and so there exists a sequence of morphisms of degree 1 from $X[i]$ to $X$ for all $i \geq 0$.
Hence there is a zig-zag between any two suspensions of $X$, and so $X$ is homologically connected.
\end{proof}

Proposition \ref{Prop:all X hom con} implies that we can apply Lemma \ref{Lem:decomposition} to $\ocC$.
That is, all objects in $\ocC$ have a hc decomposition.

%% file: Sections/ThickSubcats.tex
\section{Thick Subcategories}\label{Sec:Thick}

\subsection{Thick Subcategories of \texorpdfstring{$\cC_n$}{Cn}}

The thick subcategories of $\mathcal{C}_n$ were classified by Gratz and Zvonareva in \cite{GZ2021}.
Moreover, they prove that there exists an isomorphism of lattices between the thick subcategories of $\mathcal{C}_n$, $\mathrm{thick}(\mathcal{C}_n)$, and the lattice of \textit{non-exhaustive non-crossing partitions of} $[n] = \{1,\ldots,n\}$.

Let $\mathcal{P} = \{ B_m \subseteq [n] \mid m \in I \}$ be a collection of non-empty subsets of $[n]$, called \textit{blocks}, for some indexing set $I$.
Then Gratz and Zvonareva \cite{GZ2021} define $\mathcal{P}$ to be a \textit{non-exhaustive non-crossing partitions of} $[n]$ if $B_{m_1} \cap B_{m_2} = 0$ when $m_1 \neq m_2 \in I$, and whenever
\[
1 \leq i < k < j < l \leq n 
\]
for $i,j,k,l \in [n]$ with $i,j \in B_{m_1}$ and $k,l \in B_{m_2}$ for $m_1,m_2 \in I$, then $m_1 = m_2$.
The set of non-exhaustive non-crossing partitions of $[n]$ is denoted $NNC_n$.

The set $NNC_n$ forms a bounded lattice via the meet and join operations given by Gratz and Zvonareva in \cite{GZ2021}.
Let us fix two non-exhaustive non-crossing partitions, $\cP = \{ B_m \subseteq [n] \mid m \in I \}$ and $\cP' = \{ B'_{m'} \subseteq [n] \mid m' \in I' \}$ in $NNC_n$.
The meet of $\cP$ and $\cP'$, is given by
\[
\cP \wedge \cP' = \{ B_m \cap B_{m'} \mid m \in I, m' \in I' \, \text{and} \, B_m \cap B_{m'} \neq \emptyset\}.
\]

To describe the join, we uniquely extend $\cP \in NNC_n$ to a non-crossing partition,
\[
\overline{\cP} := \cP \cup \{ \{l\} \in [2n] \mid \text{there exists no } B \in \cP \, \text{such that } l \in B\},
\]
by adding singletons.
The join of $NNC_n$ is then given by 
\[
\cP \vee \cP' = (\overline{\cP} \vee \overline{\cP'}) \backslash \{ \{l\} \mid \text{there exists no } B \in \cP \cup \cP' \, \text{such that } l \in B\},
\]
where $\overline{\cP} \vee \overline{\cP'}$ is the join of non-crossing partitions given by Kreweras \cite{Kreweras}.

Let $\mathcal{P} = \{ B_m \subseteq [n] \mid m \in I \}$ be a non-exhaustive non-crossing partition of $[n]$.
The authors of \cite{GZ2021} consider the full subcategory $\langle \mathcal{P} \rangle \subseteq \mathcal{C}_n$ that is closed under direct sums and direct summands, and contains the zero object,
\[
\langle \mathcal{P} \rangle := \mathrm{add} \{ X \in \mathcal{C}_n \mid \ell_X = \{x,y\}, \, x,y \in \bigcup_{i \in B_m} (a_i,a_{i+1}), \text{for some} \, m \in I \}.
\]
Recall that $a_1,\ldots,a_n$ are the accumulation points in the combinatorial model for $\mathcal{C}_n$.

\begin{theorem}\cite[Theorem 3.7]{GZ2021}\label{Thm:thick in cn}
There is an isomorphism of lattices
\[
NNC_n \cong \mathrm{thick} (\mathcal{C}_n).
\]
Under this isomorphism a non-exhaustive non-crossing partition $\mathcal{P}$ corresponds to the thick subcategory $\langle \mathcal{P} \rangle$.
\end{theorem}

It is noted by Gratz and Zvonareva \cite{GZ2021} that in general for $\mathcal{P} = \{ B_m \subseteq [n] \mid m \in I \}$, the subcategory $\langle \mathcal{P} \rangle$ is equivalent to the union of mutually orthogonal thick subcategories of the form $\langle \{ B_m \} \rangle$.
Here, we show how to construct $\langle G \rangle$ for any object $G \in \mathcal{C}_n$ in terms of non-exhaustive non-crossing partitions of $[n]$.

\begin{lemma}\label{Lem:Thick objects}
    Let $F \in \mathcal{C}_n$ be a object, and let $F \cong \bigoplus_{i \in I} F_i$ be a hc decomposition of $F$.
    Let $\{B_{m_i} \mid i \in I \}$ be a collection of subsets of $[n]$ such that $\mathscr{M}_{F_i} = \bigcup_{p \in B_{m_i}} (a_p,a_{p+1})$ for all $i \in I$.
    
    Then $\mathcal{P} = \{ B_{m_i} \mid i \in I\}$ is a non-exhaustive non-crossing partition of $[n]$.
    Moreover, $\langle F \rangle$ is equivalent to $\langle \mathcal{P} \rangle$.
\end{lemma}

\begin{proof}
First, we show that the collection $\mathcal{P}$ is a non-exhaustive non-crossing partition of $[n]$.

Let $B_{m_i},B_{m_j}$ be two subsets of $[n]$, corresponding to the homologically connected objects $F_i$ and $F_j$ respectively.
As $F_i$ and $F_j$ are non-isomorphic objects in a hc decomposition of $F$, there are no morphism of degree 1 between indecomposable objects in $\langle F_i \rangle_1$ and indecomposable objects in $\langle F_j \rangle_1$.
In particular, $\sM_{F_i} \cap \sM_{F_j} = \emptyset$, hence $B_{m_i} \cap B_{m_j} = 0$ if $m_i \neq m_j$.

Let $e,f,g,h, \in [n]$ such that 
\[
1 \leq e < g < f < h \leq n,
\]
with $e,f \in B_{m_i}$ and $g,h \in B_{m_j}$.
As $F_i$ and $F_j$ are homologically connected, there exists zig-zags
\[
\begin{tikzcd}
    X_1 \arrow[r,dash] & X_2 \arrow[r,dash] & X_3 \arrow[r,dash,dashed] & X_{s-1} \arrow[r,dash] & X_s,\\
    Y_1 \arrow[r,dash] & Y_2 \arrow[r,dash] & Y_3 \arrow[r,dash,dashed] & Y_{t-1} \arrow[r,dash] & Y_t,
\end{tikzcd}
\]
with $X_1,\ldots,X_s \in \langle F_i \rangle_1$ and $Y_1,\ldots,Y_t \in \langle F_j \rangle_1$, such that $\ell_{X_1}$ has an endpoint $x \in (a_e,a_{e+1})$, $\ell_{X_s}$ has an endpoint in $x' \in (a_f,a_{f+1})$, $\ell_{Y_1}$ has an endpoint in $y \in (a_g,a_{g+1})$, and $\ell_{Y_t}$ has an endpoint in $y' \in (a_h,a_{h+1})$.
As there is a morphism of degree 1 between $X_{s'}$ and $X_{s'+1}$, then $\ell_{X_{s'}}$ and $\ell_{X_{s'+1}}$ must cross, similarly $\ell_{Y_{t'}}$ and $\ell_{Y_{t'+1}}$ must cross.

Let $\ell_X$ be an arc with endpoints $\{x,x'\}$, such that $\ell_X$ traces $\ell_{X_{s'}}$ until $\ell_{X_{s'}}$ and $\ell_{X_{s'+1}}$ cross, then $\ell_X$ traces $\ell_{X_{s'+1}}$ for all $s'=1,\ldots,s-1$.
We define $\ell_Y = \{y,y'\}$ similarly for $t'=1,\ldots,t-1$.
As $x<y<x'<y'<x$, then $\ell_X$ and $\ell_Y$ must cross, therefore $\ell_{X_{s'}}$ and $\ell_{Y_{t'}}$ must cross for some $s' \in \{1,\ldots,s\}$ and some $t' \in \{1,\ldots,t\}$.
Hence there is a morphism of degree 1 between $X_{s'} \in \langle F_i \rangle_1$ and $Y_{t'} \in \langle F_j \rangle_1$, and so $F_i \oplus F_j$ must be homologically connected.
However, $F_i$ and $F_j$ are both part of a hc decomposition of $F$, and so by Lemma \ref{Lem:decomposition}, $F_i \oplus F_j$ is only homologically connected if $i=j$.
Therefore $m_i=m_j$, and so $\mathcal{P}$ is a non-exhaustive non-crossing partition of $[n]$.

To show that $\langle F \rangle$ is equivalent to $\langle \mathcal{P} \rangle$, we note that $\langle F \rangle$ is a thick subcategory and so is equivalent to $\langle \mathcal{P}' \rangle$ for some non-exhaustive non-crossing partition $\mathcal{P}'$.
Suppose $\langle \mathcal{P}' \rangle \subset \langle \mathcal{P} \rangle$, then there exists some $p \in [n]$ such that indecomposable objects corresponding to arcs with an endpoint in $(a_p,a_{p+1})$ are in $\langle \mathcal{P} \rangle$ but not $\langle \mathcal{P}' \rangle$.
However, there exists an indecomposable direct summand of $F$ that corresponds to an arc with an endpoint in $(a_p,a_{p+1})$, as $\mathscr{M}_{F_i} = \bigcup_{p \in B_{m_i}} (a_p,a_{p+1})$ for all $i \in I$ and $\mathcal{P} = \{ B_{m_i} \mid i \in I\}$.
Therefore $F \not\in \langle \mathcal{P}' \rangle$, and so $\langle F \rangle$ is equivalent to $\langle \mathcal{P} \rangle$.
\end{proof}

Note that $\mathcal{C}_n$ satisfies the axioms of Lemma \ref{Lem:decomposition}, as $\mathcal{C}_n$ is a full triangulated subcategory of $\overline{\mathcal{C}}_n$ \cite{Paquette2020}, and so every indecomposable object is homologically connected by Lemma \ref{Prop:all X hom con}.

We provide two examples of thick subcategories of $\mathcal{C}_6$ as illustration.

\begin{figure}[H]
    \centering
    \begin{tikzpicture}
    \draw (0,0) circle (4cm);
    \draw[fill=black, opacity=0.2] (0,0) circle (4cm);
    \path[fill=white] ({4*sin(59)},{4*cos(59)}) -- ({4*sin(61)},{4*cos(61)}) -- ({4*sin(299)},{4*cos(299)}) --({4*sin(301)},{4*cos(301)}) -- cycle;
    \path[fill=white] ({4*sin(119)},{4*cos(119)}) -- ({4*sin(121)},{4*cos(121)}) -- ({4*sin(239)},{4*cos(239)}) --({4*sin(241)},{4*cos(241)}) -- cycle;
    \draw[fill=white] ({4*sin(0)},{4*cos(0)}) circle (0.1cm);
    \draw[fill=white] ({4*sin(60)},{4*cos(60)}) circle (0.1cm);
    \draw[fill=white] ({4*sin(120)},{4*cos(120)}) circle (0.1cm);
    \draw[fill=white] ({4*sin(180)},{4*cos(180)}) circle (0.1cm);
    \draw[fill=white] ({4*sin(240)},{4*cos(240)}) circle (0.1cm);
    \draw[fill=white] ({4*sin(300)},{4*cos(300)}) circle (0.1cm);
    \draw[thin] ({4*sin(74)},{4*cos(74)}) -- ({4*sin(254)},{4*cos(254)}) node[pos=0.5,above] {$\ell_X$};
    \draw[thin] ({4*sin(158)},{4*cos(158)}) .. controls (0,-3) .. ({4*sin(223)},{4*cos(223)}) node[pos=0.5,above] {$\ell_Y$};
    \draw[thin] ({4*sin(330)},{4*cos(330)}) .. controls (0,3) .. ({4*sin(30)},{4*cos(30)}) node[pos=0.5,below] {$\ell_Z$};
    \node at (4.5,0) {$\mathscr{M}_X$};
    \node at ({4.3*sin(30)},{4.3*cos(30)}) {$\mathscr{M}_Z$};
    \node at ({4.3*sin(210)},{4.3*cos(210)}) {$\mathscr{M}_Y$};
    \node at ({4.3*sin(0)},{4.3*cos(0)}) {$a_2$};
    \node at ({4.3*sin(60)},{4.3*cos(60)}) {$a_1$};
    \node at ({4.3*sin(120)},{4.3*cos(120)}) {$a_6$};
    \node at ({4.3*sin(180)},{4.3*cos(180)}) {$a_5$};
    \node at ({4.3*sin(240)},{4.3*cos(240)}) {$a_4$};
    \node at ({4.3*sin(300)},{4.3*cos(300)}) {$a_3$};
    \end{tikzpicture}
    \caption{A representation of the thick subcategory $\mathcal{T}_1$ of $\mathcal{C}_6$ containing the object $W \cong X \oplus Y \oplus Z$. The indecomposable objects correspond to the arcs that are entirely contained in the shaded area, with non-exhaustive non-crossing partition $\mathcal{P}=\{ \{1,2\},\{3,6\},\{4,5\}\}$.}
    \label{fig:thick in c6}
\end{figure}
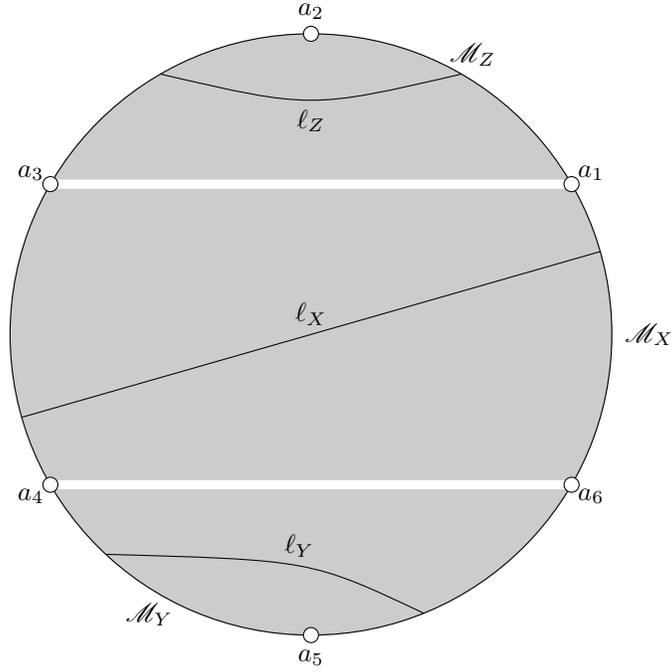

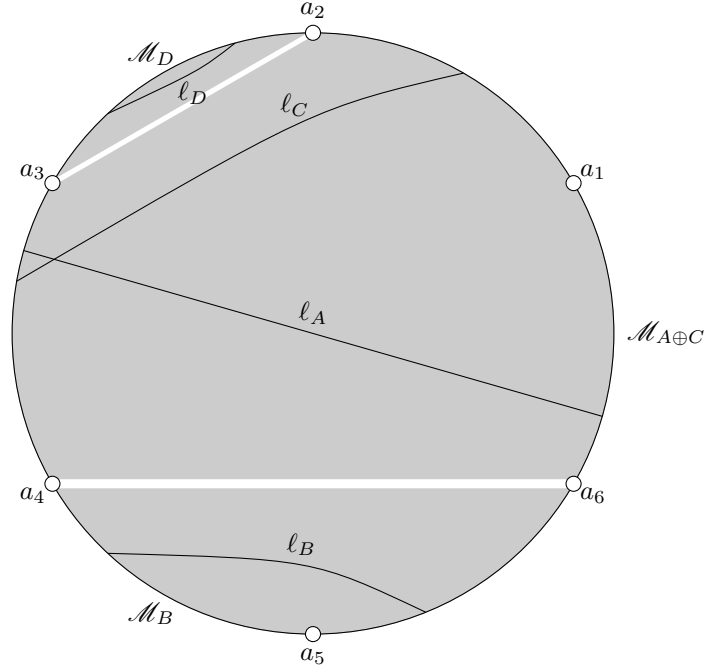
\begin{figure}[H]
    \centering
    \begin{tikzpicture}
    \draw[fill=black, opacity=0.2] (0,0) circle (4cm);
    \path[fill=white] ({4*sin(119)},{4*cos(119)}) -- ({4*sin(121)},{4*cos(121)}) -- ({4*sin(239)},{4*cos(239)}) --({4*sin(241)},{4*cos(241)}) -- cycle;
    \path[fill=white] ({4*sin(301)},{4*cos(301)}) -- ({4*sin(359)},{4*cos(359)}) -- ({4*sin(1)},{4*cos(1)}) -- ({4*sin(299)},{4*cos(299)}) -- cycle;
    \draw (0,0) circle (4cm);
    \draw[fill=white] ({4*sin(0)},{4*cos(0)}) circle (0.1cm);
    \draw[fill=white] ({4*sin(60)},{4*cos(60)}) circle (0.1cm);
    \draw[fill=white] ({4*sin(120)},{4*cos(120)}) circle (0.1cm);
    \draw[fill=white] ({4*sin(180)},{4*cos(180)}) circle (0.1cm);
    \draw[fill=white] ({4*sin(240)},{4*cos(240)}) circle (0.1cm);
    \draw[fill=white] ({4*sin(300)},{4*cos(300)}) circle (0.1cm);
    \draw[thin] ({4*sin(106)},{4*cos(106)}) -- ({4*sin(286)},{4*cos(286)}) node[pos=0.5,above] {$\ell_A$};
    \draw[thin] ({4*sin(158)},{4*cos(158)}) .. controls (0,-3) .. ({4*sin(223)},{4*cos(223)}) node[pos=0.5,above] {$\ell_B$};
    \draw[thin] ({4*sin(30)},{4*cos(30)}) .. controls (0,3) .. ({4*sin(280)},{4*cos(280)}) node[pos=0.5,above] {$\ell_C$};
    \draw[thin] ({4*sin(317)},({4*cos(317)}) .. controls (-1.5,3.5) .. ({4*sin(345)},({4*cos(345)}) node[pos=0.5,below] {$\ell_D$};
    \node at (4.7,0) {$\mathscr{M}_{A \oplus C}$};
    \node at ({4.3*sin(210)},{4.3*cos(210)}) {$\mathscr{M}_B$};
    \node at ({4.3*sin(330)},({4.3*cos(330)}) {$\sM_D$};
    \node at ({4.3*sin(0)},{4.3*cos(0)}) {$a_2$};
    \node at ({4.3*sin(60)},{4.3*cos(60)}) {$a_1$};
    \node at ({4.3*sin(120)},{4.3*cos(120)}) {$a_6$};
    \node at ({4.3*sin(180)},{4.3*cos(180)}) {$a_5$};
    \node at ({4.3*sin(240)},{4.3*cos(240)}) {$a_4$};
    \node at ({4.3*sin(300)},{4.3*cos(300)}) {$a_3$};
    \end{tikzpicture}
    \caption{A representation of the thick subcategory $\mathcal{T}_2$ of $\mathcal{C}_6$ containing the objects $E \cong A \oplus B \oplus C \oplus D$. The indecomposable objects correspond to the arcs that are entirely contained in the shaded area, with non-exhaustive non-crossing partition $\mathcal{P}' = \{ \{1,3,6\},\{2\},\{4,5\}\}$.}
    \label{fig:thicker in c6}
\end{figure}

\subsection{Thick Subcategories of \texorpdfstring{$\ocC$}{oCn}}

We may use the classification of thick subcategories in $\mathcal{C}_{2n}$ to classify the thick subcategories in $\overline{\mathcal{C}}_n$ by using the localisation functor $\pi : \mathcal{C}_{2n} \rightarrow \overline{\mathcal{C}}_n$ from \cite{Paquette2020}.

\begin{proposition}\label{Prop:Thick Subcats}
    Let $\varphi : \mathcal{C} \rightarrow \overline{\mathcal{C}}$ be a localisation functor, and let $\mathcal{T}' \subseteq \overline{\mathcal{C}}$ be a thick subcategory.
    Then $\mathcal{T}'$ is equivalent to the essential image of some thick subcategory $\mathcal{T} \subseteq \mathcal{C}$.
\end{proposition}

\begin{proof}
    Let $\mathcal{D} \subseteq \mathcal{C}$ be the subcategory defined as follows
    \[
    \mathcal{D} := \{ X \in \mathcal{C} \; \vline \; 0 \not\cong \varphi(X) \in \mathcal{T}'\},
    \]
    and let $\mathcal{T}$ be the thick closure of $\mathcal{D}$.
    We show that $\varphi(\mathcal{T}) \simeq \mathcal{T}'$.

    As $\varphi$ is an identity on objects, and $\mathcal{T}'$ is a thick subcategory then $\mathcal{D}$ is closed under direct summands.
    Now let $X,Y \in \mathcal{D}$ such that there exists a triangle 
    \[
    X \rightarrow Z \oplus Z' \rightarrow Y \rightarrow Z\wlb 1 \wrb
    \]
    with $Z \in \mathcal{T}$ but not in $\mathcal{D}$, and $Z' \in \mathcal{D}$.
    The localisation functor $\varphi$ is triangulated and therefore induces a triangle in $\overline{\mathcal{C}}_n$,
    \[
    \varphi(X) \rightarrow \varphi(Z) \oplus \varphi(Z') \rightarrow \varphi(Y) \rightarrow \varphi(X\wlb 1 \wrb).
    \]
    However, as $\varphi(X),\varphi(Y), \varphi(Z') \in \mathcal{T}'$ but $\varphi(Z) \not\in \mathcal{T}'\backslash \{0\}$, then $\varphi(Z)\cong 0$ as $\mathcal{T}'$ is thick.

    Therefore any object in $\mathcal{T}$ but not in $\mathcal{D}$ is isomorphic to a zero object in $\overline{\mathcal{C}}$, and so $\varphi(\mathcal{T}) \simeq \mathcal{T}'$.
\end{proof}

The thick subcategories of $\mathcal{C}_6$ from Figures \ref{fig:thick in c6} and \ref{fig:thicker in c6} respectively induce the following thick subcategories in $\overline{\mathcal{C}}_3$.

\begin{figure}[H]
    \centering
    \begin{tikzpicture}
        \draw (0,0) circle (4cm);
        \draw[fill=black, opacity=0.2] (0,0) circle (4cm);
        \path[fill=white] ({4*sin(89.5)},{4*cos(89.5)}) -- ({4*sin(328)},{4*cos(328)}) -- ({4*sin(329)},{4*cos(329)}) -- ({4*sin(88.5)},{4*cos(88.5)}) -- cycle;
        \path[fill=white] ({4*sin(90.5)},{4*cos(90.5)}) -- ({4*sin(212)},{4*cos(212)}) -- ({4*sin(211)},{4*cos(211)}) -- ({4*sin(91.5)},{4*cos(91.5)}) -- cycle;
        \draw[fill=black, opacity=0.2] ({4*sin(90)},{4*cos(90)}) circle (0.1cm);
        \draw[fill=black, opacity=0.2] ({4*sin(210)},{4*cos(210)}) circle (0.1cm);
        \draw[fill=black, opacity=0.2] ({4*sin(330)},{4*cos(330)}) circle (0.1cm);
        \draw ({4*sin(330)},{4*cos(330)}) circle (0.1cm);
        \draw ({4*sin(210)},{4*cos(210)}) circle (0.1cm);
        \draw ({4*sin(90)},{4*cos(90)}) circle (0.1cm);
        \draw[thin] ({3.9*sin(90)},{3.9*cos(90)}) -- ({4*sin(254)},{4*cos(254)}) node[pos=0.5,above] {$\ell_{\pi(X)}$};
    \draw[thin] ({4*sin(158)},{4*cos(158)}) .. controls (0,-3) .. ({3.9*sin(210)},{3.9*cos(210)}) node[pos=0.5,above] {$\ell_{\pi(Y)}$};
    \draw[thin] ({3.9*sin(330)},{3.9*cos(330)}) .. controls (0,3) .. ({4*sin(30)},{4*cos(30)}) node[pos=0.5,below] {$\ell_{\pi(Z)}$};
    \node at ({4.4*sin(90)},{4.4*cos(90)}) {$a_3$};
    \node at ({4.4*sin(210)},{4.4*cos(210)}) {$a_2$};
    \node at ({4.4*sin(330)},{4.4*cos(330)}) {$a_1$};
    \end{tikzpicture}
    \caption{The thick subcategory $\mathcal{T}'_1$ of $\overline{\mathcal{C}}_3$ that is equivalent to $\pi(\mathcal{T}_1)$. Again, indecomposable objects in $\mathcal{T}'_1$ are in correspondence with the arcs entirely contained in the shaded area.}
    \label{fig:thick in C3 bar}
\end{figure}
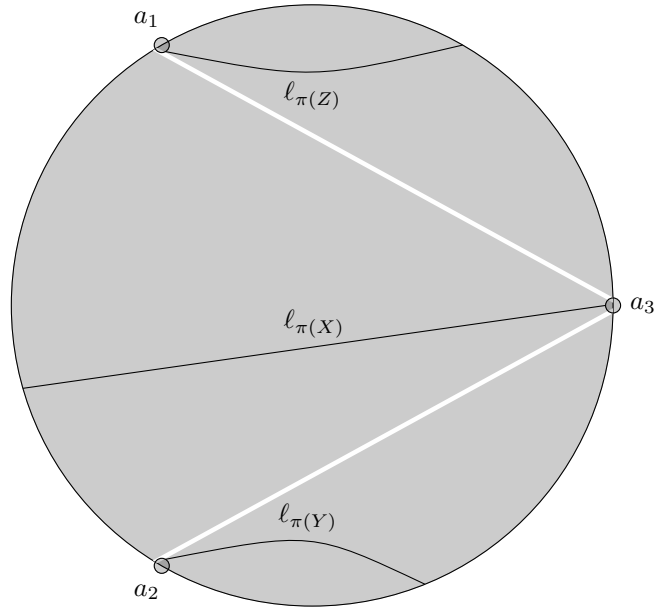

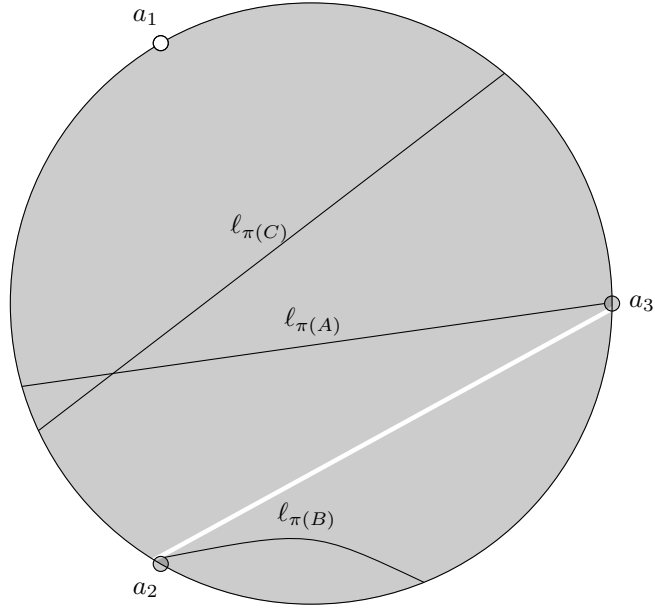
\begin{figure}[H]
    \centering
    \begin{tikzpicture}
        \draw[fill=black, opacity=0.2] (0,0) circle (4cm);
        \path[fill=white] ({4*sin(90.5)},{4*cos(90.5)}) -- ({4*sin(212)},{4*cos(212)}) -- ({4*sin(211)},{4*cos(211)}) -- ({4*sin(91.5)},{4*cos(91.5)}) -- cycle;
        \draw (0,0) circle (4cm);
        \draw[fill=black, opacity=0.2] ({4*sin(90)},{4*cos(90)}) circle (0.1cm);
        \draw[fill=black, opacity=0.2] ({4*sin(210)},{4*cos(210)}) circle (0.1cm);
        \draw[fill=white] ({4*sin(330)},{4*cos(330)}) circle (0.1cm);
        \draw ({4*sin(330)},{4*cos(330)}) circle (0.1cm);
        \draw ({4*sin(210)},{4*cos(210)}) circle (0.1cm);
        \draw ({4*sin(90)},{4*cos(90)}) circle (0.1cm);
        \draw[thin] ({3.9*sin(90)},{3.9*cos(90)}) -- ({4*sin(254)},{4*cos(254)}) node[pos=0.5,above] {$\ell_{\pi(A)}$};
    \draw[thin] ({4*sin(158)},{4*cos(158)}) .. controls (0,-3) .. ({3.9*sin(210)},{3.9*cos(210)}) node[pos=0.5,above] {$\ell_{\pi(B)}$};
    \draw[thin] ({4*sin(245)},{4*cos(245)}) -- ({4*sin(40)},{4*cos(40)}) node[pos=0.5,above] {$\ell_{\pi(C)}\:\:\:\:$};
    \node at ({4.4*sin(90)},{4.4*cos(90)}) {$a_3$};
    \node at ({4.4*sin(210)},{4.4*cos(210)}) {$a_2$};
    \node at ({4.4*sin(330)},{4.4*cos(330)}) {$a_1$};
    \end{tikzpicture}
    \caption{The thick subcategory $\mathcal{T}'_2$ of $\overline{\mathcal{C}}_3$ that is equivalent to $\pi(\mathcal{T}_2)$.}
    \label{fig:thicker in C3 bar}
\end{figure}

Recall from Definition \ref{Def:Orbits} that an object $X \in \ocC$ has an orbit in $\osM$ denoted by $\osM[X]$ (more specifically, $\ell_X$ has an orbit in $\osM$), which corresponds to the union of segments and accumulation points containing an endpoint of an arc corresponding to a direct summand of $X$.
Also recall that $X$ (again, specifically $\ell_X$) has a complete orbit in $\osM$ if $\osM[X] = \osM$.
Here, we classify the thick subcategories of $\ocC$ in terms of orbits of homologically connected objects.

\begin{lemma}\label{Prop:Mx in Mg}
Let $G$ be an object in $\ocC$, with hc decomposition $G \cong \bigoplus_{i \in I} G_i$.
Then an indecomposable object $X \in \ocC$ is in $\lang{G}$ if and only if $\osM[X] \subseteq \osM[G_i]$ for some $i \in I$.
That is, $\lang{G}$ is completely determined by the disjoint union $\bigsqcup_{i \in I} \osM[G_i]$.
\end{lemma}

\begin{proof}
The subcategory $\lang{G}$ is a thick subcategory of $\ocC$, and so by Proposition \ref{Prop:Thick Subcats} there exists a thick subcategory $\cT$ of $\cC_{2n}$ such that $\pi(\mathcal{T}) \simeq \lang{G}$.
Moreover, $\mathcal{T} \simeq \langle F \rangle$ for some object $F \in \mathcal{C}_{2n}$ by Lemma \ref{Lem:Thick objects}, where $\pi(F) \cong G$.
Therefore an indecomposable object $U \in \mathcal{C}_{2n}$ is in $\mathcal{T}$ if and only if $\ell_U$ has both endpoints in $\sM_{F_j}$, for some $F_j$ in the hc decomposition of $F$, by Theorem \ref{Thm:thick in cn} and Lemma \ref{Lem:Thick objects}.
Hence $\ell_{\pi(U)}$ has both endpoints in $\osM[\pi(F_j)] \subseteq \osM[G_i]$ for some $i \in I$ if and only if $\pi(U) \in \langle G \rangle$.
\end{proof}

\begin{definition}
    Let $\cP = \{ B_m \subseteq [n] \mid m \in I \}$ be a non-exhaustive non-crossing partition of $[n]$.
    We say that $\cP$ is an \textit{even-exclusive non-exhaustive non-crossing partition of} $[n]$ if $\{i\} \notin \cP$ for all even $i \in [n]$.
    The set of even-exclusive non-exhaustive non-crossing partitions of $[n]$ is denoted $eNNC_n$.
\end{definition}

Each even-exclusive non-exhaustive non-crossing partition of $[n]$ is by definition a non-exhaustive non-crossing partition of $[n]$.
However, we may not regard $eNNC_n$ as a sublattice of $NNC_n$ as it is not closed under meets.

\begin{example}
    The partitions $\cP_1 =\{1,2,3\}$ and $\cP_2 = \{2,4\}$ are both in $eNNC_4$, and so by extension are partitions in $NNC_4$.
    Then the meet in $NNC_4$ is given by $\cP_1 \wedge \cP_2 = \{2\}$, however $\{2\} \not\in eNNC_4$, and so cannot be the meet of $\cP_1$ and $\cP_2$ in $eNNC_4$.
\end{example}

Instead, a lattice structure may be defined on $eNNC_n$ in the following way.
Let $\eta : NNC_n \rightarrow eNNC_n$ be the map taking a partition $\cP =\{B_{m_i} \subseteq [n] \mid i \in I\} \in NNC_n$ to a partition $\cP' \in eNNC_n$, such that $\cP'=\{B_{m_i} \in \cP \mid B_{m_i} \neq \{j\} \,\text{for some even}\, j \in [n]\}$.
Then we define the meet, $\wedge_e$, in $eNNC_n$ to be the operation satisfying 
\[
\eta (\cP_1 \wedge \cP_2) := \eta (\cP_1) \wedge_e \eta (\cP_2).
\]
The join of $eNNC_n$ is the same join as in $NNC_n$, that is
\[
\cP \vee_e \cP' = \cP \vee \cP'.
\]
We denote the partial ordering of the lattice by $<_e$.

As with non-exhaustive non-crossing partitions and thick subcategories of $\cC$ in \cite{GZ2021}, we may associate a thick subcategory of $\ocC$ to each even-exclusive non-exhaustive non-crossing partition of $[2n]$.
Let $\cP \in eNNC_{2n}$, then
\begin{align*}
\lang{\cP} =& \mathrm{add}\{X \in \ocC \mid \ell_{X} = \{x,y\},\, \text{for some} \, m \in I,\\
&x,y \in \bigcup_{p\, \text{odd} \,\in B_m} (a_{(p-1)/2},a_{(p+1)/2}) \cup \bigcup _{p\, \text{even}\, \in B_m} a_{p/2}\}.
\end{align*}

Let $i \in \cP$, then the set of marked points $\{ x \in (a_{(i-1)/2},a_{(i+1)/2})\}$ if $i$ is odd, or $\{a_{i/2}\}$ if $i$ is even, is called the \textit{orbit of} $i$.

\begin{lemma}
    Let $\cP,\cQ \in eNNC_{2n}$, then $\lang{\cP} = \lang{\cQ} \subset \ocC$ if and only if $\cP = \cQ$ up to relabelling of blocks.
\end{lemma}

\begin{proof}
The statement that $\lang{\cP} = \lang{\cQ}$ if $\cP = \cQ$ follows from the definition of $\lang{\cS}$ for some $\cS \in eNNC_{2n}$.

    Now suppose $\lang{\cP} = \lang{\cQ}$, and further suppose there exists some $i \in [2n]$ such that $i \in \cP$ but $i \neq \cQ$.
    Then there exists an indecomposable object in $\lang{\cP}$ with an endpoint in the orbit of $i$, but not in $\lang{\cQ}$, a contradiction.
    Therefore $\cP$ and $\cQ$ must be equal as subsets of $[2n]$.

    Now let $i,j \in \cP,\cQ$, such that $i,j \in B_p \subseteq \cP$, but $i \in B_q$ and $j \in B_{q'}$ for $B_q \neq B_{q'} \subset \cQ$.
    There exists some indecomposable object $X \in \ocC$ such that $X$ has an endpoint in the orbit $i$, and the other endpoint in the orbit of $j$.
    Then $X \in \lang{\cP}$ but $X \neq \lang{\cQ}$, a contradiction as $\lang{\cP} = \lang{\cQ}$.
    Hence if $i,j \in B_p \subseteq \cP$, then $i,j \in B_q \subseteq \cQ$, and so $\cP = \cQ$ up to reordering.
\end{proof}

\begin{lemma}\label{Lem:Objects and Partitions 2}
    Let $F \in \ocC$ be a object, and let $F \cong \bigoplus_{i \in I} F_i$ be the hc decomposition of $F$.
    Let $\{B_{m_i} \mid i \in I \}$ be a collection of subsets of $[2n]$ such that $\osM[F_i] = \bigcup_{p\, \text{odd} \, \in B_{m_i}} (a_{(p-1)/2},a_{(p+1)/2}) \cup \bigcup _{p\, \text{even}\, \in B_{m_i}} a_{p/2}$ for all $i \in I$, where $p \in \mathbb{Z}/ n\mathbb{Z}$.
    
    Then $\mathcal{P} = \{ B_{m_i} \mid i \in I\}$ is an even-exclusive non-exhaustive non-crossing partition of $[2n]$.
    Moreover, $\langle F \rangle$ is equivalent to $\langle \mathcal{P} \rangle$.
\end{lemma}

\begin{proof}
Lemma \ref{Lem:Thick objects} shows that $\cP$ is a non-exhaustive non-crossing partition of $[2n]$.
    To see that $\cP$ is an even-exclusive non-exhaustive non-crossing partition of $[2n]$, notice that if $\{l\} \in \cP$ for $l$ even, then there exists some object $F_j$ such that $\osM[F_j] = a_j$.
    However, no object may have such an orbit in $\osM$ as it consists of a single marked point, hence $\{l\}$ cannot be in $\cP$ if $l$ is even.
\end{proof}

We now show that thick subcategories of $\ocC$ and even-exclusive non-exhaustive non-crossing partitions of $[n]$ are in bijection to each other.

\begin{theorem}\label{Thm:Lattices}
    There is an isomorphism of lattices,
    \[
    eNNC_{2n} \cong \mathrm{thick}(\ocC).
    \]
    Under this isomorphism, $\cP \in eNNC_{2n}$ is sent to the thick subcategory $\lang{\cP}$.
    Moreover, there is a commutative diagram of lattices 
    \[
    \begin{tikzcd}
        NNC_{2n} \arrow[d,"\eta",swap] \arrow[rr,"\sim"] && \mathrm{thick}(\cC_{2n}) \arrow[d,"\pi",swap]\\
        eNNC_{2n}  \arrow[rr,"\sim"] && \mathrm{thick}(\ocC) \\
        NNC_n \arrow[rr,"\sim"] \arrow[u,"\zeta"] && \mathrm{thick}(\cC_n) \arrow[u,"\xi"].
    \end{tikzcd}
    \]
\end{theorem}

\begin{proof}
First we show that there is an isomorphism of lattices between $eNNC_{2n}$ and $\mathrm{thick}(\ocC)$.

Let $\cT$ be a thick subcategory of $\ocC$, and let $F \in \ocC$ be an object such that $\cT = \lang{F}$.
Such an object $F$ exists as $\cT \subset \ocC$ is the essential image of some thick subcategory $\cT' \subset \cC_{2n}$, which is equal to $\lang{\cP'}$ for some $\cP' \in NNC_{2n}$ by Theorem \ref{Thm:thick in cn}.
Further, Lemma \ref{Lem:Thick objects} shows that there exists some object $F' \in \cC_{2n}$ such that $\lang{F'}=\lang{\cP'}$, and so 
\[
\lang{F} = \lang{\pi(F')} = \pi(\lang{F'}) = \pi(\lang{\cP'}) = \pi(\cT') = \cT.
\]

Lemma \ref{Lem:Objects and Partitions 2} tells us that there exists an even-exclusive non-exhaustive non-crossing partition $\cP$ such that $\cT = \lang{F} = \lang{\cP}$.
Therefore there exists a unique partition $\cP \in eNNC_{2n}$ for each thick subcategory $\cT \in \mathrm{thick}(\ocC)$, where uniqueness follows from $\lang{\cP} = \lang{\cQ}$ for $\cP, \cQ \in eNNC_{2n}$ if and only if $\cP = \cQ$ up to ordering.
Hence there is a bijection between $\mathrm{thick}(\ocC)$ and $eNNC_{2n}$.

Let $\cP,\cQ \in eNNC_{2n}$ be two partitions such that $\cP <_e \cQ$, and let $X \in \lang{\cP}$ be an indecomposable object, where $\ell_X = \{x,y\}$.
By definition of $\lang{P}$, $X \in \lang{P}$ if and only if $x,y \in \bigcup_{p\, \text{odd} \,\in B} (a_{(p-1)/2},a_{(p+1)/2}) \cup \bigcup _{p\, \text{even}\, \in B} a_{p/2}$ for some block $B \in \cP$.
However, as $\cP <_e \cQ$, then $B \subseteq C$ for some block $C \in \cQ$, and so $X \in \lang{\cQ}$, thus $\lang{P} \subset \lang{Q}$.
Therefore the bijection between $\mathrm{thick}(\ocC)$ and $eNNC_{2n}$ is order preserving, and so is an isomorphism of lattices.

Next we show that commutativity of the diagram of lattices 
\[
\begin{tikzcd}
    NNC_{2n} \arrow[d,"\eta",swap] \arrow[rr,"\sim"] && \mathrm{thick}(\cC_{2n}) \arrow[d,"\pi",swap]\\
    eNNC_{2n}  \arrow[rr,"\sim"] && \mathrm{thick}(\ocC) \\
    NNC_n \arrow[u,"\zeta"] \arrow[rr,"\sim"] && \mathrm{thick}(\cC_n) \arrow[u,"\xi"],
\end{tikzcd}
 \]
 where the horizontal isomorphisms are given by taking a partition $\cQ$ to the thick subcategory $\lang{\cQ}$.
 
 Let $\cP' \in NNC_{2n}$ be a partition, and let $\eta (\cP') \in eNNC_{2n}$.
 We wish to show that $\lang{\cP} = \pi(\lang{\cP'})$, where by abuse of notation $\pi \colon \mathrm{thick}(\cC_{2n}) \rightarrow \mathrm{thick}(\ocC)$ takes a thick subcategory $\cT' \subset \cC_{2n}$ to the subcategory $\pi(\cT') \subseteq \ocC$.
 Let $B$ be a block in $\cP'$ such that $B \neq \{i\}$ for some even $i \in [2n]$, then $B = \eta(B) \in \eta(\cP')$, and so $\lang{\eta(B)} = \pi(\lang{B}) \subset \ocC$.
This is because an indecomposable object $Z' \in \cC_{2n}$ is in $\lang{B}$ if and only if $\ell_{Z'} = \{z,z'\}$ with $z,z' \in \bigcup_{i \in B} (a_i,a_{i+1})$, and so $\pi(Z') \neq 0$ if $z,z' \in \bigcup_{i\, \text{odd} \,\in B} (a_{(i-1)/2},a_{(i+1)/2}) \cup \bigcup _{i\, \text{even}\, \in B} a_{i/2}$, which is exactly when an object is in $\lang{\eta(B)}$.
 
 Now let $B'$ be a block in $\cP'$ such that $B' = \{i\}$ for some even $ i \in [2n]$, then $\eta(B')=0$, so $\lang{\eta(B')}=0$.
 Also, $\lang{B'} \subseteq \cD \subset \cC_{2n}$, where $\cD$ is the thick subcategory in Construction \ref{Cons:Verdier Localisation}, such that $\ocC := \cC_{2n} / \cD$.
 This is because all indecomposable objects in $\lang{B'}$ correspond to a short arc with endpoints in $(a_{i},a_{i+1})$, and as $i$ is even, then those indecomposable objects are also in $\cD$.
 Hence $\pi(\lang{B'}) = 0 =\lang{\eta(B')}$.
 Therefore $\lang{\eta(\cP')} = \pi(\lang{\cP'})$ for all $\cP' \in NNC_{2n}$, and so the upper square of the diagram commutes.

Let $\zeta$ act on a partition $\cP \in NNC_n$ by taking a block $B = \{i_1,i_2,\ldots,i_l\} \in \cP$ to $\zeta(B) = \{2i_1-1,2i_2-1,\ldots,2i_l-1\} \in \cQ \in eNNC_{2n}$.
This will always be a partition in $eNNC_{2n}$ as all elements in $\zeta(B)$ are odd.
The map $\xi$ is the inclusion map from $\mathrm{thick}(\cC_n)$ to $\mathrm{thick}(\ocC)$.
It is clear that there are no even numbers in any $\zeta(B)$, and so an indecomposable object $Z \in \ocC$ is in $\lang{\zeta(B)}$ if and only if $\ell_Z$ has endpoints in $\bigcup_{2i-1 \in \zeta(B)} (a_i,a_{i+1})$.
That is, all indecomposable objects in $\lang{\zeta(B)}$ correspond to either a long arc or a short arc.

Similarly, an indecomposable object in $\lang{B} \subset \cC_n$ corresponds to an arc with endpoints in $\bigcup_{i \in B} (a_i,a_{i+1})$, and so via inclusion, an indecomposable object is in $\xi(\lang{B})$ if and only if the corresponding arc has endpoints in $\bigcup_{i \in B} (a_i,a_{i+1})$.
Hence an object $Z \in \ocC$ is in $\lang{\zeta(B)}$ if and only if $Z \in \xi(\lang{B})$, and so the lower square of the diagram commutes.
\end{proof}

\begin{example}
    Let $\cT'_1$ be the thick subcategory from Figure \ref{fig:thick in C3 bar}, and let $\cP \in eNNC_{2n}$ such that $\lang{\cP} = \cT'_1$, then $\cP = \{ \{1,2\},\{3,6\},\{4,5\}\}$.
    For $\cT'_2$ in Figure \ref{fig:thicker in C3 bar}, and let $\cQ \in eNNC_{2n}$ be the partition such that $\lang{\cQ} = \cT'_2$, then $\cQ = \{\{1,3,6\},\{4,5\}\}$.
\end{example}

\subsection{Counting Thick Subcategories}

The number of non-exhaustive non-crossing partitions of $[n]$ was given by Gratz and Zvonareva \cite{GZ2021} as
\[
\lvert NNC_n \rvert = \sum_{i=0}^n {n \choose i} \cdot C_i,
\]
where $C_i$ is the $i^{\text{th}}$ Catalan number.
Here we compute the number of even-exclusive non-exhaustive non-crossing partitions of $[2n]$ in terms of $\lvert NNC_m \rvert$, for $n \leq m \leq 2n$.

\begin{lemma}
    The number of even-exclusive non-exhaustive non-crossing partitions of $[2n]$ is given by the formula,
    \[
    \lvert eNNC_{2n} \rvert = \sum_{j=0}^{n} (-1)^j {n \choose j} \cdot \lvert NNC_{2n-j} \rvert.
    \]
\end{lemma}

\begin{proof}
    Every even-exclusive non-exhaustive non-crossing partition of $[2n]$ is also a non-exhaustive non-crossing partitions of $[2n]$, and so $\lvert eNNC_{2n} \rvert \leq \lvert NNC_{2n} \rvert$.
    We prove our claim by starting with all partitions in $NNC_{2n}$, and systematically removing every partition which is not in $eNNC_{2n}$.
    We label the even numbers between $1$ and $2n$ inclusive, by $a_{i_1},a_{i_2},\ldots,a_{i_n}$ in no particular order.

    In general, there are $\lvert NNC_{2n-j} \rvert$ partitions of $NNC_{2n}$ that contain the subsets $\{a_{i_1}\}, \{a_{i_2}\}, \ldots , \{a_{i_j}\}$, and ${n \choose j}$ choices for $a_{i_1},a_{i_2},\ldots,a_{i_j}$.
    To see this, consider a partition $\cP$ such that $a_{i_1},a_{i_2},\ldots,a_{i_j} \not\in \cP$, then $\cP$ is equivalent to a non-exhaustive non-crossing partition of $[2n-j]$, of which there are $\lvert NNC_{2n-j} \vert$ possible partitions.
    Then we may form the partition $\cP'$ as the union of $\cP$ and $\{a_{i_1}\}, \{a_{i_2}\}, \ldots , \{a_{i_j}\}$, this is still a non-exhaustive non-crossing partition of $[2n]$, but not in $eNNC_{2n}$.
    Therefore there are at most ${n \choose j} \cdot \lvert NNC_{2n-j} \rvert$ partitions of $NNC_{2n}$ that contain $j$ subsets consisting of a single even number.
    If $j=n$, then there are exactly $\lvert NNC_n \rvert$ partitions of $[2n]$ containing the subsets $\{a_1\},\ldots,\{a_n\}$.

    However, we have over counted the possible number of partitions in $NNC_{2n}$ but not in $eNNC_{2n}$, as, for example, a partition containing $\{a_1\},\{a_2\}$ is also counted once as a partition containing $\{a_1\}$ and once as a partition containing $\{a_2\}$.
    We claim that this means we must in fact add $\sum_{j=1}^{n} (-1)^j {n \choose j} \cdot \lvert NNC_{2n-j} \rvert$ to $\lvert NNC_{2n} \rvert$, which we show via induction on $j$.
    
    For $j=1$, we have not previously removed any partitions, so we remove all partitions containing $\{a_{i_1}\}$ for some $a_{i_1}$, and so add $(-1)^1 {n \choose 1} \cdot \lvert NNC_{2n-1} \rvert$.
    For $j=2$, every partition containing $\{a_{i_1}\},\{a_{i_2}\}$ has been removed twice, once for containing $\{a_{i_1}\}$ and once for containing $\{a_{i_2}\}$, hence we must add back a copy of all such partitions.
    Therefore we add $(-1)^2 {n \choose 2} \cdot \lvert NNC_{2n-2} \rvert$.

    Suppose it is true for $j=l$, and let $\cP_1$ be a partition in $NNC_{2n}$ containing $\{a_{i_1}\}\{a_{i_2}\}\ldots \{a_{i_{l+1}}\}$.
    Then there are $\sum_{m=0}^l (-1)^m {l+1 \choose m}$ copies of $\cP$ counted in $\sum_{m=0}^{l} (-1)^m {n \choose l} \cdot \lvert NNC_{2n-l} \rvert$.
    However, for all $l > 0$, we have 
    \[
    \sum_{m=0}^l (-1)^m {l+1 \choose m} = (-1)^l.
    \]
    Therefore we must add the partition $\cP_1$ back in $(-1)^{l+1}$ times so that $\cP_1$ is counted a total of zero times, which is what we want as $\cP_1 \not\in eNNC_{2n}$.
    As there there are at least ${n \choose l+1} \cdot \lvert NNC_{2n-l-1} \rvert$ partitions of $NNC_{2n}$ that contain $l+1$ blocks of a single even number, we therefore add $(-1)^{l+1} {n \choose l+1} \cdot \lvert NNC_{2n-l-1} \rvert$ partitions.
    Hence 
    \[
    \lvert eNNC_{2n} \rvert = \sum_{j=0}^{n} (-1)^j {n \choose j} \cdot \lvert NNC_{2n-j} \rvert.
    \]
\end{proof}

\begin{remark}
The sequence of values of $\lvert eNNC_{2n} \rvert$ for $n \geq 1$ is equal to the diagonal of the Euler-Seidel matrix for the Catalan numbers \cite{Barry}.
\end{remark}

%% file: Sections/Generators.tex
\section{Generators of \texorpdfstring{$\ocC$}{Cn}}\label{Sec:GensofCn}

\subsection{Generators}

Here we look at the classical generators of $\ocC$, and provide necessary and sufficient conditions for an object to be a classical generator.

\begin{proposition}\label{Prop:Gen is HomCon}
Let $G$ be a generator of $\ocC$.
Then $G$ is homologically connected.
\end{proposition}

\begin{proof}
Suppose that $G \in \ocC$ is not homologically connected.
We show that $G$ cannot be a generator of $\ocC$.

By Lemma \ref{Lem:decomposition}, there exists a hc decomposition of $G$,
\[
G \cong \bigoplus_{i \in I} G_i.
\]
Further, Lemma \ref{Prop:Mx in Mg} tells us that an object $X \in \ocC$ is in $\langle G \rangle$ if and only if $\osM[X] \subseteq \osM[G_i]$ for some $i \in I$.
Now let $Y \in \ocC$ correspond to the arc $\ell_Y = \{y_1,y_2\}$, such that $y_1 \in \osM[G_j]$ and $y_2 \in \osM[G_{j'}]$ with $j \neq j' \in I$.
Then $\osM[Y] \not\subset \osM[G_i]$ for any $i \in I$, hence $Y \not\in \langle G \rangle$, so $G$ is not a classical generator of $\ocC$.
\end{proof}

Finally, we can combine Lemma \ref{Prop:Mx in Mg} and Proposition \ref{Prop:Gen is HomCon} to classify all of the generators of $\ocC$, and moreover, show that they all must be strong generators too.

\begin{theorem}\label{Thm:GensOfCn}
Let $G$ be an object in $\ocC$, then $G$ is a generator of $\ocC$ if and only if $G$ is homologically connected and $G$ has a complete orbit of $\osM$.
\end{theorem}

\begin{proof}
Let $G$ be a generator, then by Proposition \ref{Prop:Gen is HomCon} $G$ is homologically connected, and by Proposition \ref{Prop:Mx in Mg} $\osM[X] \subseteq \osM[G]$ for all $X \in \ocC$, and so $G$ has a complete orbit in $\osM$.

Let $G$ be homologically connected and have a complete orbit in $\osM$, then by Proposition \ref{Prop:Mx in Mg} all indecomposable objects are in $\langle G \rangle$, and so $G$ is a generator of $\ocC$.
\end{proof}

It follows from Theorem \ref{Thm:GensOfCn} that no short arcs may be direct summands of a minimal strong generator of $\ocC$ for all $n$.

\begin{corollary}\label{Cor:NoShortArcs}
Let $\ell_X$ be a short arc.
Then $X$ cannot be a direct summand of a minimal strong generator of $\ocC$.
\end{corollary}

\begin{proof}
Let $\ell_X$ be a short arc, and $G$ be a generator with $X$ as a direct summand, and let $F$ be an object such that $G \cong F \oplus X$.

Let $G$ have no other direct summands with endpoints in the same segment as $\ell_X$.
Then $\Ext{i}{X}{F}=0$ for all $i \in \mathbb{Z}$, and so $G$ is not homologically connected.

Now let $Y$ be an indecomposable direct summand of $F$ such that $\osM[X] \subsetneq \osM[Y]$.
If $\ell_Y$ has only one endpoint on the segment shared by $\ell_X$, then $\ell_Y$ is a long arc or limit arc, and thus there exists a triangle
\[
Y\wlb j \wrb \rightarrow X \oplus Z \rightarrow Y\wlb l \wrb \rightarrow Y[j+1]
\]
and so $X \in \langle Y \rangle$ and so $F$ is also a generator of $\ocC$.

If $\osM[X] = \osM[Y]$, then Proposition \ref{Prop:Mx in Mg} implies that $X \in \langle Y \rangle$, as Proposition \ref{Prop:all X hom con} means that $Y$ must be homologically connected.

Therefore, $G$ cannot be a minimal generator if it has a short arc as an indecomposable direct summand.
\end{proof}

\begin{lemma}\label{Lem:Loops aren't minimal}
    Let $G$ be a generator of $\ocC$, and suppose there exists a zig-zag in $\langle G \rangle_1$,
    \[
\begin{tikzcd}
M_1 \arrow[r,dash] & M_2 \arrow[r,dash,dashed] & M_d \arrow[r,dash]& M_{d+1} \arrow[r,dash] & M_1,
\end{tikzcd}
\]
such that $M_1 \not\cong M_i \wlb j \wrb$ for all $i=2, \ldots, d+1$ and $j \in \mathbb{Z}$, and $\ell_{M_1}$ shares an endpoint each with $\ell_{M_2 \wlb a_2 \wrb}$ and $\ell_{M_{d+1} \wlb a_{d+1} \wrb}$ for some $a_2,a_{d+1} \in \mathbb{Z}$.
Then the object $F$ such that $G \cong F \oplus M_1$ is a generator of $\ocC$.
\end{lemma}

\begin{proof}
    As $\ell_{M_1}$ shares an endpoint each with $\ell_{M_2 \wlb a_2 \wrb}$ and $\ell_{M_{d+1} \wlb a_{d+1} \wrb}$, then $F$ has a complete orbit in $\osM$.
    We now need to show that $F$ is homologically connected.

    Let $\ell_{M_i \wlb a_i \wrb}=\{x_i,y_i\}$, where $x_i < x_{i+1} \leq y_i < y_{i+1} < x_i$ and $x_{i+1}=y_i$ if and only if $y_i$ is an accumulation point.
    That is, for some $m \in \osM$ such that $x_1 < m < y_1 < x_1$, then $x_i \leq m  \leq y_i < x_i$ for some $i=2,\ldots,d+1$, again with equality if and only if $m$ is an accumulation point.
    
    Suppose an arc $\ell_N = \{n_1,n_2\}$ crosses $\ell_{M_1}$, and $N$ is in $\langle G \rangle_1$.
    Then $\ell_N$ must also cross some arc $\ell_{M_i}$ for $i = 2,\ldots,d+1$, or share an endpoint with $\ell_{M_i}$ at an accumulation point.
    In either case, there exists a zig-zag of length 1 between $N$ and $M_i$, and so $F$ is homologically connected.

    Therefore, Theorem \ref{Thm:GensOfCn} tells us that $F$ is a generator of $\ocC$.
\end{proof}

%% file: Sections/GenTime.tex
\subsection{Rouquier Dimension}

The notion of a dimension on a triangulated category was introduced by Rouquier in \cite{Rouquier} as a tool to help study the representation dimension of a finite dimensional algebra.
Given an algebra $A$, Rouquier provides a series of lower bounds for various dimensions of $A$, whenever $A$ satisfies a set of given properties.
Notably, they provide a lower bound on the representation dimension of a finite dimensional algebra over a field, which is the dimension of the bounded derived category of said algebra.
This allows them to provide the first known examples of algebras with representation dimension $>3$, a long standing question at the time.

Throughout the rest of this thesis, we shall use Rouquier dimension to refer to the dimension of a triangulated category, with the latter term being preferred in \cite{Rouquier}.

\begin{definition}\label{Def:Rdim}
Let $\cC$ be a triangulated category.
If $G \in \cC$ is a generator, then we define the \textit{generation time of} $G$ to be the minimal integer $m$ such that $\langle G \rangle_{m+1} = \cC$.
The set of generation times of generators of $\cC$ is called the \textit{Orlov spectrum}, denoted $\mathcal{O(C)}$, and the infimum of $\mathcal{O(C)}$ is called the \textit{Rouquier dimension} of $\cC$, denoted $\rdim{\cC}$.

If there exists no such $G$, then we say $\rdim{\cC}$ is $\infty$.
\end{definition}

The following two results may be considered more folklore, however both have proofs found in \cite{ElaginLunts} by Elagin and Lunts.

\begin{lemma}\label{Lem:Rdim=0}
Let $\cT$ be a Krull-Schmidt, triangulated category.
Then $\rdim{\cT}=0$ if and only if $\cC$ contains only finitely many indecomposables up to isomorphism and shifts.
\end{lemma}

Now we fix an object $E \in \ocC$ for all $n \geq 1$.
This is a classical generator of $\ocC$, and generates $\ocC$ in a single step.

Let $X_i \in \overline{\mathcal{C}}_n$ correspond to the arc $\ell_{X_i}=\{a_1,z_i\}$ for all $i=1,\ldots,n$, and $Y_j \in \overline{\mathcal{C}}_n$ correspond to the arc $\ell_{Y_j}=\{a_1,a_{j+1}\}$ for $j=1,\ldots, n-1$.
Then we let $E$ be the direct sums of all $X_i$'s and $Y_j$'s, i.e.\
\[
E = (\bigoplus_{i=1}^n X_i) \oplus (\bigoplus_{j=1}^{n-1} Y_j).
\]

\begin{figure}[h]
\centering
\begin{tikzpicture}
\draw (0,0) circle (4cm);
\draw[fill=white] (0,4) circle (0.1cm);
\draw[fill=white] ({4*sin(144)},{4*cos(144)}) circle (0.1cm);
\draw[fill=white] ({4*sin(72)},{4*cos(72)}) circle (0.1cm);
\draw[fill=white] ({4*sin(288)},{4*cos(288)}) circle (0.1cm);
\draw[fill=white] ({4*sin(216)},{4*cos(216)}) circle (0.1cm);
\draw[thin] (0,3.9) .. controls (0.5,3.3) .. ({4*sin(36)},{4*cos(36)}) node[pos=0.85, below] {$\ell_{X_n}$};
\draw[thin] (0,3.9) .. controls (1,0.3) .. ({4*sin(108)},{4*cos(108)}) node[pos=0.75, below] {$\ell_{X_{n-1}}$};
\draw[thin] (0,3.9) .. controls (-0.5,3.3) .. ({4*sin(324)},{4*cos(324)}) node[pos=0.85, below] {$\ell_{X_1}$};
\draw[thin] (0,3.9) .. controls (-1,0.3) .. ({4*sin(252)},{4*cos(252)}) node[pos=0.75, below] {$\ell_{X_2}$};
\draw[thin] (0,3.9) .. controls (1.5,1.5) .. ({3.9*sin(72)},{3.9*cos(72)}) node[pos=0.75, below] {$\ell_{Y_{n-1}}$};
\draw[thin] (0,3.9) .. controls (0.5,0) .. ({3.9*sin(144)},{3.9*cos(144)}) node[pos=0.85, below] {$\ell_{Y_{n-2}}\; \; \;$};
\draw[thin] (0,3.9) .. controls (-1.5,1.5) .. ({3.9*sin(288)},{3.9*cos(288)}) node[pos=0.75, below] {$\ell_{Y_1}$};
\draw[thin] (0,3.9) .. controls (-0.5,0) .. ({3.9*sin(216)},{3.9*cos(216)}) node[pos=0.85, below] {$ \; \; \ell_{Y_2}$};
\draw[thick,dotted] ({3.5*sin(190)},{3.5*cos(190)}) -- ({3.5*sin(170)},{3.5*cos(170)});
\path[fill=white] ({3.9*sin(190)},{3.9*cos(190)}) -- ({4.1*sin(190)},{4.1*cos(190)}) -- ({4.1*sin(170)},{4.1*cos(170)}) -- ({3.9*sin(170)},{3.9*cos(170)}) -- cycle;
\node at ({4.4*sin(0)},{4.4*cos(0)}) {$a_1$};
\node at ({4.4*sin(288)},{4.4*cos(288)}) {$a_2$};
\node at ({4.4*sin(216)},{4.4*cos(216)}) {$a_3$};
\node at ({4.4*sin(144)},{4.4*cos(144)}) {$a_{n-1}$};
\node at ({4.4*sin(72)},{4.4*cos(72)}) {$a_n$};
\node at ({4.4*sin(324)},{4.4*cos(324)}) {$z_1$};
\node at ({4.4*sin(252)},{4.4*cos(252)}) {$z_2$};
\node at ({4.4*sin(36)},{4.4*cos(36)}) {$z_n$};
\node at ({4.5*sin(108)},{4.5*cos(108)}) {$z_{n-1}$};
\end{tikzpicture}
\caption{The arcs corresponding to $X_i$'s and $Y_j$'s.}
\label{fig:XandYs}
\end{figure}
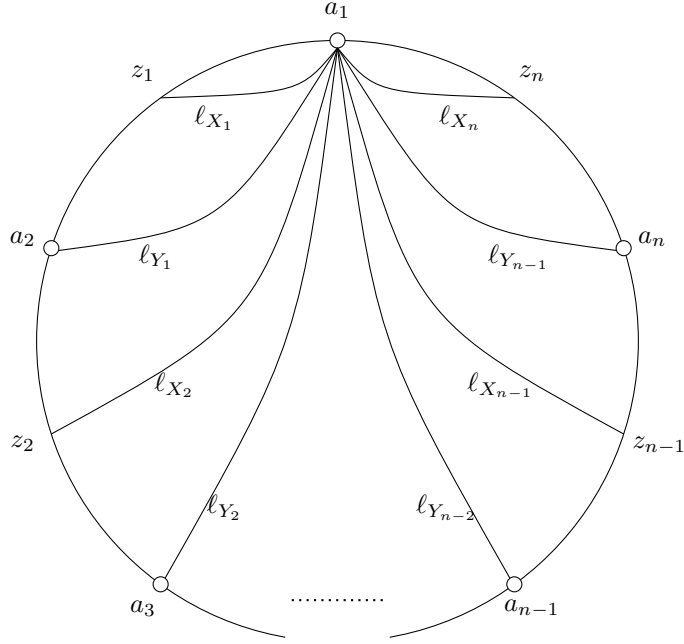

\begin{proposition}\label{Prop:E1}
An indecomposable object is in $\langle E \rangle_1$ if and only if it corresponds to an arc of the form $(a_1,y)$, for $a_1 \neq y \in \osM$.
Thus $E$ has a complete orbit in $\osM$.
\end{proposition}

\begin{proof}
Let $M \in \langle E \rangle_1$ be indecomposable, then by definition $M$ is a direct summand of $E \wlb p \wrb$ for some $p \in \mathbb{Z}$.
This means that either $M \cong X_i[a]$ for some $i \in \{1,\ldots, n\},\; a \in \mathbb{Z}$, or $M \cong Y_j[b]$ for some $j \in \{1, \ldots, n-1\},\; b \in \mathbb{Z}$.

If $M \cong X_i[a]$, then it corresponds to the arc $\{a_1, z_i - a\}$ by \cite{Paquette2020}.
If $M \cong Y_j[b]$, then it corresponds to the arc $\{a_1, a_{j+1}\}$, also by \cite{Paquette2020}.
Therefore, when $M \in \langle E \rangle_1$ is indecomposable, $\ell_M$ is of the form $\{a_1, y\}$ for $a_1 \neq y \in \osM$.

Now let $\ell_N=\{a_1,y\}$ for some $a_1 \neq y \in \osM$.
If $y \in L(\mathscr{M})$, then $N \cong Y_j$ for some $1 \leq j \leq n-1$.
Similarly, if $y \in \mathscr{M}$, then $N \cong X_i[a]$ for some $1 \leq i \leq n$ and $a \in \mathbb{Z}$.
In both cases, $N$ is isomorphic to an indecomposable summand of a suspension of $E$, and therefore $N \in \langle E \rangle_1$.
\end{proof}

\begin{lemma}\label{Lem:E is a gen}
    The object $E$ is a generator of $\overline{\mathcal{C}}_n$.
    Moreover, $E$ has generation time 1.
\end{lemma}

\begin{proof}
    By Proposition \ref{Prop:E1}, $E$ has a complete orbit in $\osM$.
    Also, by \cite{Paquette2020}, for any two arcs $\ell_U$ and $\ell_V$ sharing an endpoint at an accumulation point, then $\Ext{1}{U}{V} \cong k$ if and only if $\ell_U$ is an anti-clockwise rotation of $\ell_V$ about the shared endpoint.
    Therefore there exists a zig-zag of length 1 between all indecomposable summands of $E$, and so $E$ is homologically connected.
    It follows from Theorem \ref{Thm:GensOfCn} that $E$ is a generator of $\ocC$.

    Let $X \in \ocC$ be an indecomposable object with corresponding arc $\ell_X = \{x_1,x_2\}$, and suppose $x_1 < x_2 \leq a_1 < x_1$.
    If $x_2 = a_1$, then $X \in \lang[1]{E}$ by Proposition \ref{Prop:E1}.
    Else, if $x_2 \neq a_1$, then there exists a triangle
    \[
    U \rightarrow X \rightarrow V \rightarrow U \wb{1},
    \]
    where $U,V$ are indecomposable objects and correspond to the arcs $\ell_U = \{a_1,x_1\}$ and $\ell_V = \{a_1,x_2\}$.
    Therefore $X \in \lang[2]{E}$, and so any indecomposable object is in $\lang[2]{E}$, which is closed under direct sums.
    Hence $E$ has generation time 1.
\end{proof}

The following is a consequence of Lemmas \ref{Lem:Rdim=0} and \ref{Lem:E is a gen}.

\begin{corollary}\label{Cor:Rdim}
The Rouquier dimension of $\ocC$ is 1.
\end{corollary}

\subsection{Orlov Spectrum}

The Orlov spectrum, then known as the dimension spectrum, of a triangulated category was first introduced by Orlov in \cite{OrlovSpectra}.
Orlov looks at the dimension spectra of bounded derived categories of various geometric categories, and asks the question of whether the Orlov spectrum of the bounded derived category of coherent sheaves on a smooth, quasi-projective scheme form an integer interval?
It proves to be a difficult task to compute the Orlov spectrum for triangulated categories in general, and when it may not be possible to directly compute the Orlov spectrum, it may be natural to ask whether or not there are upper and lower bounds to the Orlov spectrum.
We provide an interval bound for the Orlov spectrum of $\ocC$ in Theorem \ref{Thm:Orlov}, and compute the Orlov spectrum for $\ocC[1]$.

Ballard, Favero and Katzarkov study the Orlov spectra of triangulated categories arising from mirror symmetry in \cite{Orlovbounds}.
They develop techniques to associate a generator to any given object in the bounded derived category of coherent sheaves on a smooth Calabi-Yau hypersurface, and show that these generators are uniformly bounded in their generation time.
More relevantly to our situation, they also compute the Orlov spectrum of the bounded derived category of the category of finitely generated modules of the path algebra of a quiver of Dynkin type $A_n$.

\begin{theorem}\cite{Orlovbounds}
    Let $Q$ be a quiver of Dynkin type $A_n$.
    Then the Orlov spectrum of $\mathrm{D}^{\mathrm{b}}(\mathrm{mod}\, kQ)$ is equal to the integer interval $\{0,\ldots,n-1\}$.
\end{theorem}

Our results begin by looking at a minimal strong generator of $\ocC$.

\begin{lemma}\label{Lem:GenTime}
Let $G$ be a minimal strong generator of $\ocC$.
Then the homological length of $G$ is an upper bound for the generation time of $G$.
\end{lemma}

\begin{proof}
Let $G$ have homological length $d$.
We show that any object in $\ocC$ has a generation time at most $d$.

Let $X \in \ocC$ such that $\ell_X = \{x_1,x_2\}$.
As $G$ is a generator, then $x_1,x_2 \in \osM[G]$ by Theorem \ref{Thm:GensOfCn}, and so there exists $G_1,G_{d+1} \in \langle G \rangle_1$ with corresponding arcs $\ell_{G_1} = \{x_1,z_1\}$ and $\ell_{G_{d+1}} = \{y_{d+1},x_2\}$.
Moreover, $G$ is homologically connected, so there exists a minimal zig-zag 
\[
\begin{tikzcd}
    G_1 \arrow[r,dash] & G_2 \arrow[r,dash] & \cdots \arrow[r,dash] & G_d \arrow[r,dash] & G_{d+1}
\end{tikzcd}
\]
with length $d$, and all $G_i \in \langle G \rangle_1$.
Let $\ell_{G_i} = \{ y_i,z_i\}$ for all $i=2,\ldots,d$.
We claim that $X \in \langle G \rangle_{d+1}$.

By Lemma \ref{Lem:Reduce minimal}, there exists a series of zig-zags of the form
\[
\begin{tikzcd}
    M_i \arrow[r,dash] & G_{i+1} \arrow[r,dash] & \cdots \arrow[r,dash] & G_d \arrow[r,dash] & G_{d+1}
\end{tikzcd}
\]
with $M_i \in \langle G \rangle_i$, and $\ell_{M_i} = \{x_1,z_i\}$.
Importantly, there exists a zig-zag
\[
\begin{tikzcd}
    M_d \arrow[r,dash] & G_{d+1}
\end{tikzcd}
\]
with $M_d \in \langle G \rangle_d$.

The above zig-zag means that there is a morphism of degree 1 between $M_d$ and $G_{d+1}$, and so at least one of the following is a distinguished triangle
\[
\begin{tikzcd}[row sep = tiny]
    M_d \arrow[r] & A \arrow[r] & G_{d+1} \arrow[r] & M_d\wlb 1 \wrb,\\
    G_{d+1} \arrow[r] & B \arrow[r] & M_d \arrow[r] & G_{d+1} \wlb 1 \wrb.
\end{tikzcd}
\]
If both are distinguished triangles, then the object $X$ must be a direct summand of either $A$ or $B$, as $\ell_{M_d} = \{x_1,z_d\}$, $\ell_{G_{d+1}} = \{y_{d+1}, x_2\}$ and $\ell_X = \{x_1,x_2\}$.
If only one of them is a distinguished triangle, then by \cite[Prop. 3.14]{Paquette2020}, the arcs $\ell_{M_d}$ and $\ell_{G_{d+1}}$ share an endpoint at an accumulation point.
Hence $X$ is isomorphic to either $M_d$ or $G_{d+1}$, and therefore in $\lang[d]{G}$, or $X$ is isomorphic to the middle term of the distinguished triangle.

Therefore $X \in \langle G \rangle_{d+1}$, and therefore $G$ has generation time at most $d$.
\end{proof}

Consequently, we find the Orlov spectrum of $\ocC[1]$.

\begin{corollary}\label{Cor:Orlov C1}
    The Orlov spectrum of $\ocC[1]$ is
    \[
    \cO(\ocC[1]) = \{1\}.
    \]
\end{corollary}

\begin{proof}
    We show that there is no minimal generator $G$ with homological length $d \geq 2$, and so the generation time of $G$ has an upper bound of 1.
    Let 
    \[
    \begin{tikzcd}
        G_1 \arrow[r,dash] & G_2 \arrow[r,dash,dashed] & G_d \arrow[r,dash] & G_{d+1}
    \end{tikzcd}
    \]
    be a zig-zag with objects in $\langle G \rangle_1$.
    Corollary \ref{Cor:NoShortArcs} tells us that all $\ell_{G_i}$ must be limit arcs as $G$ is minimal.
    However, by \cite[Prop. 3.14]{Paquette2020} there is a non-trivial $\mathrm{Ext}^1$-space between indecomposable objects corresponding to limit arcs that share an endpoint at an accumulation point, which all $\ell_{G_i}$ do as there is only a single accumulation point for $\ocC[1]$.
    Hence there is a minimal zig-zag
    \[
    \begin{tikzcd}
        G_1 \arrow[r,dash]  & G_{d+1},
    \end{tikzcd}
    \]
    and so the homological length of any minimal generator $G$ is 1, therefore the upper bound on the generation time of $G$ is also 1.
\end{proof}

We now look at the homological length of minimal generators in $\ocC$ in an effort to find a bound for the Orlov spectrum of $\ocC$.
To do this, we show that for any integer $d$ up to a given value, there exists a minimal generator with homological length equal to $d$.

\begin{proposition}\label{Prop:1 to 2n-2}
Let $1 \leq d \leq 2n-2$ be an integer.
Then there exists a generator $M$ of $\ocC$ such that $M$ has homological length $d$.
\end{proposition}

\begin{proof}
We construct a generator $M_d$ for each $1 \leq d \leq 2n-2$ using induction on the generator $E$.

We know that $E$ has homological length 1 by Lemma \ref{Lem:E is a gen}, so we construct a new object from $E$, called $M_2$ by replacing $Y_1$ with the object $Z_1$, corresponding to the arc $\ell_{Z_1}=\{z_1,a_2\}$.
It is clear that $M_2$ is homologically connected and has a complete orbit in $\osM$, and so by Theorem \ref{Thm:GensOfCn} we know $M_2$ is a generator of $\ocC$.
We construct $M_3$ from $M_2$ by replacing $X_2$ with the object $Z_2$, corresponding to the arc $\ell_{Z_2}=\{a_2,z_2\}$, and again we see that $M_3$ is a generator by Theorem \ref{Thm:GensOfCn}.
We repeat this construction for all $M_d$, $1 \leq d \leq 2n-2$.

Assume that $M_d$ has homological length $d$, we show that $M_{d+1}$ has homological length $d+1$.
By Lemma \ref{Lem:GenTime} we know that there exists at least one minimal zig-zag between indecomposable summands of $M_d$ of length $d$.
We have a minimal zig-zag of the form
\[
\begin{tikzcd}
H \arrow[r,dash,"f_1"] & X_1 \arrow[r,dash,dashed] & Z_1 \arrow[r,dash,"f_{d-1}"]& Z_{d-1} \arrow[r,dash,"f_d"] & Z_d,
\end{tikzcd}
\]
where $H \not\cong X_1$ is in both $\langle E \rangle_1$ and $\langle M_d \rangle_1$.
However, when we replace an indecomposable summand, say $N$, in $M_d$ with $Z_{d+1}$, we only have length 1 minimal zig-zags between $Z_{d+1}$ and $Z_d$, and all other minimal zig-zags between $Z_{d+1}$ and another indecomposable summand of $M_{d+1}$ contain $Z_d$ up to suspension.
Hence we get a minimal zig-zag
\[
\begin{tikzcd}
H \arrow[r,dash,"f_1"] & X_1 \arrow[r,dash,dashed] & Z_1 \arrow[r,dash,"f_d"]& Z_d \arrow[r,dash,"f_{d+1}"] & Z_{d+1},
\end{tikzcd}
\]
which has length $d+1$.

To see that we get no other new minimal zig-zags of length $l>d+1$ between indecomposable summands of $M_{d+1}$, consider that any minimal zig-zag in $\lang{M_d}$ containing $N$ in the middle will not be minimal as any sequence in a zig-zag of the form
\[
\begin{tikzcd}
X_1 \wlb j \wrb \arrow[r,dash] & N \arrow[r,dash] & H
\end{tikzcd}
\]
can be reduced to a sequence of the form
\[
\begin{tikzcd}
X_1 \wlb j \wrb \arrow[r,dash] & H
\end{tikzcd}
\]
by \cite{Paquette2020}.
Hence a minimal zig-zag containing $N$ may only contain $N$ either at the start or end of the zig-zag, and so by removing $N$ as a summand of $M_{d+1}$ means that no minimal zig-zag of this form has an increased length.

Therefore we see that the homological length of $M_{d+1}$ is $d+1$.
Thus, by induction, we see that there exists some object with homological length $d$ for all $d=1,\ldots,2n-2$.
\end{proof}

Here we see an example of one of the minimal strong generators constructed in Proposition \ref{Prop:1 to 2n-2}.

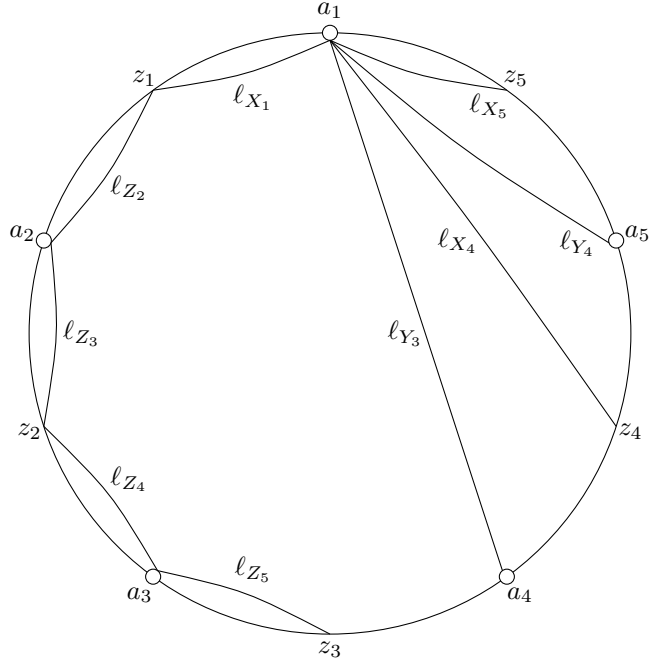
\begin{figure}[h]
\centering
\begin{tikzpicture}
\draw (0,0) circle (4cm);
\draw[fill=white] (0,4) circle (0.1cm);
\draw[fill=white] ({4*sin(72)},{4*cos(72)}) circle (0.1cm);
\draw[fill=white] ({4*sin(144)},{4*cos(144)}) circle (0.1cm);
\draw[fill=white] ({4*sin(216)},{4*cos(216)}) circle (0.1cm);
\draw[fill=white] ({4*sin(288)},{4*cos(288)}) circle (0.1cm);
\draw[thin] (0,3.9) .. controls ({3.6*sin(342)},{3.6*cos(342)}) .. ({4*sin(324)},{4*cos(324)});
\draw[thin] (0,3.9) .. controls ({3.6*sin(18)},{3.6*cos(18)}) .. ({4*sin(36)},{4*cos(36)});
\draw[thin] (0,3.9) .. controls ({2.4*sin(54)},{2.4*cos(54)}) .. ({4*sin(108)},{4*cos(108)});
\draw[thin] (0,3.9) .. controls ({3*sin(36)},{3*cos(36)}) .. ({3.9*sin(72)},{3.9*cos(72)});
\draw[thin] (0,3.9)--({3.9*sin(144)},{3.9*cos(144)});
\draw[thin] ({3.9*sin(288)},{3.9*cos(288)}) .. controls ({3.6*sin(306)},{3.6*cos(306)}) .. ({4*sin(324)},{4*cos(324)});
\draw[thin] ({3.9*sin(288)},{3.9*cos(288)}) .. controls ({3.6*sin(270)},{3.6*cos(270)}) .. ({4*sin(252)},{4*cos(252)});
\draw[thin] ({3.9*sin(216)},{3.9*cos(216)}) .. controls ({3.6*sin(234)},{3.6*cos(234)}) .. ({4*sin(252)},{4*cos(252)});
\draw[thin] ({3.9*sin(216)},{3.9*cos(216)}) .. controls ({3.6*sin(198)},{3.6*cos(198)}) .. ({4*sin(180)},{4*cos(180)});
\node at (0,4.3) {$a_1$};
\node at ({4.3*sin(72)},{4.3*cos(72)}) {$a_5$};
\node at ({4.3*sin(144)},{4.3*cos(144)}) {$a_4$};
\node at ({4.3*sin(216)},{4.3*cos(216)}) {$a_3$};
\node at ({4.3*sin(288)},{4.3*cos(288)}) {$a_2$};
\node at ({4.2*sin(36)},{4.2*cos(36)}) {$z_5$};
\node at ({4.2*sin(108)},{4.2*cos(108)}) {$z_4$};
\node at ({4.2*sin(180)},{4.2*cos(180)}) {$z_3$};
\node at ({4.2*sin(252)},{4.2*cos(252)}) {$z_2$};
\node at ({4.2*sin(324)},{4.2*cos(324)}) {$z_1$};
\node at ({3.3*sin(342)},{3.3*cos(342)}) {$\ell_{X_1}$};
\node at ({3.3*sin(306)},{3.3*cos(306)}) {$\ell_{Z_2}$};
\node at ({3.3*sin(270)},{3.3*cos(270)}) {$\ell_{Z_3}$};
\node at ({3.3*sin(234)},{3.3*cos(234)}) {$\ell_{Z_4}$};
\node at ({3.3*sin(198)},{3.3*cos(198)}) {$\ell_{Z_5}$};
\node at ({3.7*sin(35)},{3.7*cos(35)}) {$\ell_{X_5}$};
\node at ({2.1*sin(54)},{2.1*cos(54)}) {$\ell_{X_4}$};
\node at ({3.5*sin(70)},{3.5*cos(70)}) {$\ell_{Y_4}$};
\node at (1,0) {$\ell_{Y_3}$};
\end{tikzpicture}
\caption{The minimal strong generator $M_5$ in $\ocC[5]$.}
\label{fig:Md example}
\end{figure}

Finally, we show that there exists no minimal strong generator of $\ocC$ with a homological length greater than $2n-2$, and so compute an upper bound for the Orlov spectrum of $\ocC$.

\begin{theorem}\label{Thm:Orlov}
The Orlov spectrum of $\ocC$ for $n \geq 2$ is bounded above by $2n-2$.
That is
\[
\cO(\ocC) \subseteq \{1,\ldots,2n-2\}.
\]
\end{theorem}

\begin{proof}
By Proposition \ref{Prop:1 to 2n-2} we know that there is a generator with homological length $2n-2$ and so has generation time at most $2n-2$ by Lemma \ref{Lem:GenTime}, so we only need to show that there exists no minimal strong generator with homological length greater than $2n-2$.
To show this, we consider two situations, one where some generator $M$ has $\geq 2n$ non-isomorphic indecomposable direct summands, and one where $M$ has $< 2n$ non-isomorphic indecomposable direct summands.

Suppose $M$ has $\geq 2n$ non-isomorphic indecomposable direct summands.
Then if we consider each segment and accumulation point as a vertex, and each arc as an edge, then may construct each generator as a graph with $2n$ vertices.
Basic results from graph theory tell us that if we have $\geq 2n$ edges on $2n$ vertices, then we must have a loop somewhere in the graph, and this loop then corresponds to a zig-zag of the form
\[
\begin{tikzcd}
M_1\wlb a_1 \wrb \arrow[r,dash] & M_2\wlb a_2 \wrb \arrow[r,dash,dashed] & M_d\wlb a_d \wrb \arrow[r,dash]& M_{d+1}\wlb a_{d+1} \wrb \arrow[r,dash] & M_1\wlb a_1 \wrb,
\end{tikzcd}
\]
such that $\ell_{M_i \wlb a_i \wrb}$ either crosses $\ell_{M_{i-1}\wlb a_{i-1}\wrb }$ and $\ell_{M_{i+1} \wlb a_{i+1} \wrb}$, or shares an endpoint at an accumulation point, and $\ell_{M_{d+1}\wlb a_{d+1} \wrb}$ crosses $\ell_{M_1\wlb a_1 \wrb}$ or shares an endpoint at an accumulation point.
Lemma \ref{Lem:Loops aren't minimal} tells us then that the object $M'$ such that $M \cong M' \oplus M_1$ is also a generator, and so $M$ is not a minimal generator.

Now suppose that $M$ has $2n-1$ non-isomorphic indecomposable direct summands.
If $M$ has homological length $2n-1$, then there must exist a minimal zig-zag of length $2n-1$, let this zig-zag be
\[
\begin{tikzcd}
M_1 \arrow[r, dash] & M_2\wlb m_2 \wrb \arrow[r,dash,dashed] & M_{2n-1}\wlb m_{2n-1} \wrb \arrow[r,dash] & M_i\wlb j \wrb.
\end{tikzcd}
\]
There exists some subsequence of this zig-zag
\[
\begin{tikzcd}
M_1 \arrow[r, dash] & M_2\wlb m_2 \wrb \arrow[r,dash,dashed] & M_{i-1}\wlb m_{i-1} \wrb \arrow[r,dash] & M_i\wlb m_i \wrb,
\end{tikzcd}
\]
of length $i-1 < 2n-1$.
Suppose that there exists no morphisms of degree 1 in either direction between $M_i\wlb m_i \wrb$ and $M_i\wlb j \wrb$, then $\ell_{M_i}$ is a short arc, and so $M$ is not a minimal generator by Corollary \ref{Cor:NoShortArcs}.

Now suppose that there does exist a morphism of degree 1 between $M_i\wlb m_i \wrb$ and $M_i\wlb j \wrb$, then we have a zig-zag
\[
\begin{tikzcd}
M_1 \arrow[r, dash] & M_2\wlb m_2 \wrb \arrow[r,dash,dashed] & M_{i-1}\wlb m_{i-1} \wrb \arrow[r,dash] & M_i\wlb m_i \wrb \arrow[r,dash] & M_i\wlb j \wrb,
\end{tikzcd}
\]
of length $i \leq 2n-1$, where the length is equal when $M_i \cong M_{2n-1}$.
As we only need to consider zig-zags of length $ \geq 2n-1$ we may assume that $M_i \cong M_{2n-1}$, and we need to show that the zig-zag
\begin{equation}\label{Eq:longzigzag}
\begin{tikzcd}[column sep=1.2em]
M_1 \arrow[r, dash] & M_2\wlb m_2 \wrb \arrow[r,dash,dashed] & M_{2n-2}[m_{2n-2}] \arrow[r,dash] & M_{2n-1}\wlb m_{2n-1} \wrb \arrow[r,dash] & M_{2n-1}\wlb j \wrb,
\end{tikzcd}
\end{equation}
is not minimal.

Suppose there are two arcs $\ell_{M_i}$ and $\ell_{M_{i'}}$ that share $a \in L(\mathscr{M})$ as an endpoint.
If $i' \neq i \pm 1$, then \eqref{Eq:longzigzag} is not a minimal zig-zag.
Therefore, suppose that $i'=i+1$, then there exists a zig-zag
\[
\begin{tikzcd}
M_i\wlb m_i \wrb \arrow[r,dash] & M_{i+1}\wlb m'_{i+1} \wrb
\end{tikzcd}
\]
such that we have the zig-zag
\[
\begin{tikzcd}
M_1 \arrow[r, dash] & M_2\wlb m_2 \wrb \arrow[r,dash,dashed] & M_{2n-2}\wlb m'_{2n-2} \wrb \arrow[r,dash] & M_{2n-1}\wlb m'_{2n-1} \wrb \cong M_{2n-1}\wlb j \wrb,
\end{tikzcd}
\]
which has length $2n-2$, and so the minimal zig-zag between $M_1$ and $M_{2n-1}\wlb j \wrb$ is at most length $2n-2$.

Now suppose that every $a_l \in L(\mathscr{M})$ is the endpoint of exactly one arc $\ell_{M_l}$, for $M_l$ indecomposable direct summands of $M$.
Then between the $n$ segments of $\mathscr{M}$, there are at least $n$ long arcs corresponding to direct summands of $M$.
However, again via seeing the segments as vertices and the long arcs as edges, there must be a cycle in the induced graph, and so $M$ is not minimal as we could remove any of the arcs in the cycle, and the resulting object would still have a complete orbit and be homologically connected, and so $M$ is not a minimal generator.
Hence there are no minimal strong generators of $\ocC$ with a homological length greater than $2n-2$.
Thus $2n-2$ is an upper bound of the Orlov spectrum of $\ocC$.
\end{proof}